\documentclass{article}
\usepackage{graphicx,overpic}
\usepackage[left=3cm, top=2cm, right=3cm, bottom=3cm]{geometry}
\usepackage{listings}
\usepackage{float}
\usepackage{amssymb,amsmath,mathtools}
\usepackage{enumerate}
\usepackage[numbers,square]{natbib}
\usepackage{wrapfig}
\usepackage{framed}
\usepackage{tabstackengine}
\usepackage{braket}
\usepackage[nottoc]{tocbibind}\settocbibname{References}
\usepackage[title]{appendix}
\usepackage{indentfirst}
\usepackage{setspace}
\usepackage[none]{hyphenat}
\usepackage{xcolor}

\usepackage{tikz}
\tikzset{every loop/.style={}}
\usetikzlibrary{arrows,shapes,decorations.markings,automata,backgrounds,petri,bending,calc}

\usepackage{fancyhdr}
\fancypagestyle{firststyle}
{\fancyfoot[l]{
	Shepherd was funded by the Engineering and Physical Sciences Research Council. Gardam was funded by the Deutsche Forschungsgemeinschaft (DFG, German Research Foundation) under Germany's Excellence Strategy EXC 2044 –390685587, Mathematics Münster: Dynamics–Geometry–Structure. Woodhouse was supported by the Israel Science Foundation (grant 1026/15).}
	\fancyfoot[c]{1}

}

\usepackage{tikz-cd} 
\tikzset{
    labl/.style={anchor=south, rotate=90, inner sep=.5mm}
}

\usepackage[utf8]{inputenc}
\usepackage[english]{babel} 

\usepackage{amsthm}
\newtheorem{thm}{Theorem}[section]
\newtheorem{lem}[thm]{Lemma}
\newtheorem{prop}[thm]{Proposition}

\numberwithin{equation}{section}

\newtheorem{letterthm}{Theorem}[section]

\theoremstyle{definition}
\newtheorem{defn}[thm]{Definition} 
\newtheorem{exmp}[thm]{Example} 
\newtheorem{remk}[thm]{Remark}

\newenvironment{claim}[1]{\par\noindent\underline{Claim:}\space#1}{}
\newenvironment{claimproof}[1]{\par\noindent\underline{Proof:}\space#1}{\leavevmode\unskip\penalty9999 \hbox{}\nobreak\hfill\quad\hbox{$\blacksquare$}}

\newcommand{\calM}{\mathcal{M}}
\newcommand{\calS}{\mathcal{S}}

\newcommand{\calG}{\mathcal{G}}
\newcommand{\calC}{\mathcal{C}}

\makeatletter
\newenvironment{breakproof}[1][\proofname]{\par
  \pushQED{\qed}%
  \normalfont \topsep6\p@\@plus6\p@\relax
  \trivlist
  \item[\hskip\labelsep
        \itshape
    #1\@addpunct{.}]\ignorespaces\item
}{%
  \popQED\endtrivlist\@endpefalse
}
\makeatother

\usepackage{mathrsfs}

\newcommand{\N}{\mathbb{N}}

\newcommand{\wt}[1]{\widetilde{#1}}

\DeclareMathOperator{\Aut}{Aut}

\DeclareMathOperator{\Stab}{Stab}

\DeclareMathOperator{\star*}{star}
\DeclareMathOperator{\LCM}{LCM}
\newcommand{\acts}{\curvearrowright}

\begin{document}

	\begin{center}
{\LARGE\bf
Two Generalisations of Leighton's Theorem}\\
\bigskip
\bigskip
{\large Sam Shepherd\footnote{Shepherd was funded by the Engineering and Physical Sciences Research Council.}\\
with an appendix by Giles Gardam\footnote{Gardam was funded by the Deutsche Forschungsgemeinschaft (DFG, German Research Foundation) under Germany's Excellence Strategy EXC 2044 –390685587, Mathematics Münster: Dynamics–Geometry–Structure.} and Daniel J. Woodhouse\footnote{Woodhouse was supported by the Israel Science Foundation (grant 1026/15).}}

\end{center}
\bigskip
\begin{abstract}
	Leighton's graph covering theorem says that two finite graphs with a common cover have a common finite cover. We present a new proof of this using groupoids, and use this as a model to prove two generalisations of the theorem. The first generalisation, which we refer to as the symmetry-restricted version, restricts how balls of a given size in the universal cover can map down to the two finite graphs when factoring through the common finite cover - this answers a question of Neumann from \cite{Neumann}. Secondly, we consider covers of graphs of spaces (or of more general objects), which leads to an even more general version of Leighton's Theorem. We also compute upper bounds for the sizes of the finite covers obtained in Leighton's Theorem and its generalisations. An appendix by Gardam and Woodhouse provides an alternative proof of the symmetry-restricted version, that uses Haar measure instead of groupoids.
\end{abstract}
\bigskip
\tableofcontents
\bigskip
\section{Introduction}

\begin{letterthm}(Leighton's Theorem)\label{introLeighton}\\
	Let $G_1$ and $G_2$ be finite connected graphs with a common cover. Then they have a common finite cover. 
\end{letterthm}

This theorem was proven for $k$-regular graphs by Angluin and Gardener \cite{Angluin}, and the general case was proven by Leighton \cite{Leighton}. An alternative proof was given by Bass and Kulkarni \cite{BassKulkarni} using Bass--Serre theory. Walter Neumann revisited both proofs in the context of coloured graphs, and investigated generalisations to ``symmetry-restricted graphs'' \cite{Neumann}. In this paper we solve Leighton's Theorem for Neumann's notion of symmetry-restricted graph. 

There have been several applications of Leighton-type theorems in recent years. Levitt used the original Leighton's Theorem to solve the commensurability problem for certain generalised Baumslag--Solitar groups \cite{Levitt}. Behrstock--Neumann employed Neumann's symmetry-restricted version of Leighton's Theorem to prove quasi-isometric rigidity for certain non-geometric 3-manifold groups \cite{3mlfd}. Woodhouse proved a version of Leighton's Theorem for ``graphs with fins'', and used this to prove pattern rigidity for free groups with line patterns \cite{Woodhouse}; and the construction of the finite cover from that paper was recently used by the first author and Woodhouse to prove quasi-isometric rigidity for certain graphs of virtually free groups with two-ended edge groups \cite{ShepherdWoodhouse}. And Stark and Woodhouse used the symmetry-restricted Leighton's Theorem from this paper to prove action rigidity for free products of hyperbolic manifold groups \cite{StarkWoodhouse}.

We formulate symmetry-restriction by working in the universal cover. Indeed, if $T$ is a tree that covers finite graphs $G_1$ and $G_2$ via maps $p_i:T\to G_i$, then Leighton's Theorem gives us a common finite cover $G$ of $G_1$ and $G_2$, and by elementary covering space theory we can draw the following commutative diagram of covering maps, where $g$ is an automorphism of $T$.

\begin{equation*}
\begin{tikzcd}[
ar symbol/.style = {draw=none,"#1" description,sloped},
isomorphic/.style = {ar symbol={\cong}},
equals/.style = {ar symbol={=}},
subset/.style = {ar symbol={\subset}}
]
T\ar{dd}[swap]{p_1}\ar{rr}{g}\ar{dr}&&T\ar{dl}\ar{dd}{p_2}\\
&G\ar{dl}\ar{dr}\\
G_1&&G_2
\end{tikzcd}
\end{equation*}

If the coverings $p_i:T\to G_i$ have deck transformation groups $\Gamma_i$, then this diagram implies that the conjugate $\Gamma_1^g$ is commensurable to $\Gamma_2$ in $\Aut(T)$, because their intersection is the deck transformation group of the (right-hand) covering $T\to G$. Note that for $\Gamma\leqslant\Aut(T)$, $\Gamma$ being the deck transformation group of a covering $T\to G$ of a finite graph is equivalent to $\Gamma$ acting freely cocompactly on $T$, which is also equivalent to $\Gamma$ being a free uniform lattice in $\Aut(T)$ (with respect to the compact-open topology on $\Aut(T)$). So Leighton's Theorem is equivalent to saying: for $T$ a tree, and $\Gamma_1,\Gamma_2\leqslant\Aut(T)$ free uniform lattices, there exists $g\in\Aut(T)$ with $\Gamma_1^g$ commensurable to $\Gamma_2$. However we have no control over the conjugating element $g$, and this is exactly the issue addressed by symmetry-restriction. We need the following definition, which, as was later pointed out to us, is the same as \cite[Definition 3.1]{Banks}.

\begin{defn}\label{defn:SymClosure}
	Let $T$ be a tree, let $H\leqslant\Aut(T)$, and let $R$ be an integer. For a vertex $v\in V(T)$, we denote the $R$-ball centred at $v$ by $B_R(v)$. Given some $g\in\Aut(T)$ and a vertex $v\in V(T)$, we let $g_v:B_R(v)\to B_R(gv)$ be the restriction of $g$ to $B_R(v)$. We define the \emph{$R$-symmetry-restricted closure} of $H$ to be: 
	
	$$	\mathscr{S}_R(H) := \{ g \in \Aut(T)\mid \forall v \in V(T), \exists h \in H \text{ s.t. }h_v=g_v:B_R(v)\to B_R(gv)\}
	$$ 
\end{defn}

The Symmetry-restricted Leighton's Theorem can then be stated as follows. Neumann solved the special case where $R=1$ and $T/H$ is a tree \cite[Theorem 2.4]{Neumann}.

\begin{letterthm}(Symmetry-restricted Leighton's Theorem)\\\label{thm:SymLeighton} 
	Let $T$ be a tree, and $H\leqslant\Aut(T)$, and let $\Gamma_1, \Gamma_2 \leqslant H$ be free uniform lattices in $\Aut(T)$.
	Then for all $R \in \mathbb{N}$ there exists $g \in \mathscr{S}_R(H)$ such that $\Gamma_1^g$ is commensurable to $\Gamma_2$ in $\Aut(T)$. 
\end{letterthm}

This theorem was proven by the first author, and independently by Gardam and Woodhouse. The proof of Gardam and Woodhouse, which uses the Haar measure on subgroups of Aut$(T)$, is given in the appendix. The particular way we have stated the theorem here is also due to Gardam and Woodhouse. The proof of the first author is given in Section \ref{sec:SymVersion}, and deduces it from a more general version of Leighton's Theorem, concerning what we call graphs of objects, which we now go on to describe. A simple motivating example is the following:

\begin{exmp}\label{exmp:graphoftriangles}
	Define a \emph{graph of polygons} to be a space consisting of solid regular polygons with some edges joining vertices of the polygons. A \emph{covering} of graphs of polygons is a topological covering that restricts to isometries between polygons. Leighton's Theorem holds for graphs of polygons: if two finite graphs of polygons (ie. with finite underlying graphs) are covered by the same tree of polygons, then they have a common finite cover. This can be deduced from Theorem \ref{thm:SymLeighton}, but it also follows from a more general version of Leighton's Theorem as we describe below - both deductions are given in Proposition \ref{prop:polygonLeighton}.
	\bigskip
	\begin{figure}[H]
		\centering
	\begin{tikzpicture}[node distance = 4cm,
	scale=3,
	triangle/.style = {fill=red, draw=red, regular polygon, regular polygon sides=3, minimum size=40pt},
	square/.style = {fill=red, draw=red, regular polygon, regular polygon sides=4, minimum size=40pt},
	pentagon/.style = {fill=red, draw=red, regular polygon, regular polygon sides=5, minimum size=40pt},
	every loop/.style={min distance=2cm},
	]
	\clip (-1,-1.8) rectangle (3,0.3);
	
	\node[triangle] (1) {};
	\node[pentagon] (2) at (2,0) {};
	\node[square] (3) at (1,-1) {};
	
	\path
	(1.corner 3) edge (3.corner 2)
	(1.corner 2) edge [bend right]  (3.corner 2)
	(1.corner 1) edge (2.corner 1)
	(2.corner 3) edge [loop below] (2.corner 3)
	;
	\end{tikzpicture}
	\caption{A graph of polygons.}
	\end{figure}
\end{exmp}
\bigskip

Graphs of polygons can be thought of as graphs of spaces where the vertex spaces are polygons and the edge spaces are points. It is natural to ask if Leighton's Theorem also holds for more complicated graphs of spaces. One difficulty is that a general covering between graphs of spaces induces coverings between the vertex spaces rather than isometries. In fact one can exploit this difficulty to construct a pair of finite graphs of spaces with a common universal cover but no common finite cover; this can be done with Baumslag--Solitar groups or with a graph of graphs due to Wise, which are discussed further in Section \ref{sec:GraphofObjects}. The other difficulty that arises is when the isometry group of an edge space is infinite. A way of solving this for many examples is by working in a different category of spaces; graphs of polygons are defined in the category of metric spaces, and the edge spaces are points so have trivial isometry groups, but for other examples one might need to work in the category of simplicial complexes and consider automorphism groups of edge spaces rather than isometry groups. 

To deal with both problems we define a \emph{graph of objects} to be a graph of spaces with respect to a given category of spaces, and we define \emph{coverings of graphs of objects} to be coverings of graphs of spaces that restrict to isomorphisms between edge and vertex spaces. We also define a covering from a graph of objects $X$ to itself to be an \emph{automorphism} of $X$ if it induces an automorphism of the underlying graph, and we denote the group of automorphisms by $\Aut(X)$. If $X$ is a graph of objects and $X_e$ is an edge space, then we have a homomorphism from the stabiliser $\Aut(X)_e$ of the edge $e$ to $\Aut(X_e)$, the automorphism group of $X_e$ with respect to the given category, and we call the image the \emph{isotropy group of $e$ in $\Aut(X)$}. Precise definitions for all these notions are given in Section \ref{sec:GraphofObjects}. Our version of Leighton's Theorem for graphs of objects is then as follows.

\begin{letterthm}(Graph of Objects Leighton's Theorem)\\\label{thm:ObjectLeighton}
	Let $X^1$ and $X^2$ be finite graphs of objects covered by a tree of objects $X$. If $\Aut(X)$ has finite edge isotropy groups, then $X^1$ and $X^2$ have a common finite cover.
\end{letterthm}

We actually prove a stronger version of this theorem which incorporates a notion of symmetry-restriction for graphs of objects (Theorem \ref{thm:SymObjectLeighton}), and it is this stronger version that we use to deduce Theorem \ref{thm:SymLeighton}.

A key ingredient in the proof of the Graph of Objects Leighton's Theorem is the use of groupoids. In general, groupoids are very natural objects that arise in a wide range of mathematical contexts, and they are especially powerful in topology for stitching together local information in a basepoint free manner into a global structure. Some basic background on groupoids is given in Section \ref{sec:groupoids}. Our groupoid method also gives a new proof of the original Leighton's Theorem for graphs, which we explain in Section \ref{sec:Original}.

The groupoid proof of the original Leighton's Theorem is very instructive for understanding the proof of the graph of objects version. A sketch of the proof is as follows. Let $T$ be a tree that covers our finite graphs $G_1$ and $G_2$ via maps $p_1:T\to G_1$ and $p_2:T\to G_2$. One idea is to consider how each star (of a vertex) in $T$ maps down to $G_1$ and $G_2$, and observe that there are only finitely many stars up to $(p_1,p_2)$-invariant isomorphism. It is then natural to try and piece together these (isomorphism classes of) stars to build a common finite cover of $G_1$ and $G_2$. But two stars can only be joined along an edge if they are compatible on this edge, meaning that the maps to $G_1$ and $G_2$ agree on this edge. Unfortunately this approach doesn't quite work, even if we take duplicates of some stars, as there may not be any ``symmetry'' in the compatibility relations between stars. The solution  is to create symmetry by expanding our collection of stars, with the aid of a finite groupoid $\calS$ consisting of maps between stars in $G_1$ and $G_2$. For a vertex $v\in V(T)$, an automorphism $g\in\Aut(T)$ restricts to a map $\star*(v)\to\star*(gv)$, and this induces a map $\star*(p_i(v))\to\star*(p_j(v))$ for a choice of $i,j\in\{1,2\}$. Such maps form the groupoid $\calS$. To each $s\in\calS$ that maps from a star in $G_1$ to a star in $G_2$, we associate a new star (that is not yet part of a graph) with maps down to $G_1$ and $G_2$ that commute with $s$ - these will form our collection of stars that we can piece together to form the common finite cover of $G_1$ and $G_2$.
\bigskip

In Section \ref{sec:Bounds} we compute upper bounds for the sizes of the finite covers obtained in Theorems \ref{introLeighton} and \ref{thm:SymObjectLeighton}, and a bound for the index of commensurability in Theorem \ref{thm:SymLeighton}. These bounds come from analysing the constructions used in the proofs of these theorems, except for Theorem \ref{introLeighton} it turns out we get a better bound if we use Leighton's original proof. There is no reason to believe these bounds are best possible. The bounded version of Theorem \ref{introLeighton} is the following.

\begin{letterthm}(Bounded Leighton's Theorem)\label{introboundL}\\
	Let $G_1$ and $G_2$ be finite connected graphs. Set $E:=\tfrac{1}{2}|E(G_1)|$ and $V:=|V(G_2)|$. Then $G_1$ and $G_2$ have a common cover $G$ such that $|V(G)|\leq 2V\exp(2\sqrt{E\log{E}})$.
\end{letterthm}

\textbf{Acknowledgements:}\,
I would like to thank my supervisor, Martin Bridson, for his advice while I was developing the ideas for this paper, and for his thorough proofreading of the final drafts. I would like to thank Ric Wade for helpful comments, which in particular spurred me to produce Theorem \ref{thm:SymObjectLeighton}. I would also like to thank the referee for their useful suggestions.

\bigskip
\section{Graphs and groupoids}\label{sec:groupoids}

In this section we define graphs, stars and covers, which are used throughout the paper, and we give some background about groupoids that will be needed in our proofs of Theorems \ref{introLeighton} and \ref{thm:ObjectLeighton}.

\begin{defn}(Graphs)\\
	A \textit{graph} $G$ is defined by the following data:
	\begin{itemize}
		\item A vertex set $V(G)$.
		\item An edge set $E(G)$.
		\item Maps $\partial_0,\partial_1:E(G)\to V(G)$ to denote the initial and terminal vertex of each edge.
		\item An involution $E(G)\to E(G)$, $e\mapsto\overline{e}$, which denotes the inversion of an edge, such that $e$ and $\overline{e}$ are always distinct, and such that $\partial_0e=\partial_1\overline{e}$ and $\partial_0\overline{e}=\partial_1e$ for any $e\in E(G)$.
	\end{itemize}
	Note that $\partial_1$ is redundant if $\partial_0$ and edge inversion have already been defined. For $A\subset E(G)$ we will use the notation $\overline{A}:=\{\overline{e}\mid e\in A\}$.\par
	A \textit{graph morphism} $\alpha:G_1\to G_2$ is given by maps $\alpha:V(G_1)\to V(G_2)$ and $\alpha:E(G_1)\to E(G_2)$ that preserve the graph structure given by $\partial_0, \partial_1$ and edge inversion. Note that it is enough to check that $\partial_0$ and edge inversion are preserved. A graph morphism $G\to G$ that is bijective on edge and vertex sets is called an \textit{automorphism}, and the group of automorphisms of a graph $G$ is denoted Aut$(G)$.
\end{defn}

\begin{defn}(Stars and covers)\\
	Let $G$ be a graph. For $v\in V(G)$ define the \textit{star} of $v$ by
	\begin{equation*}
	\text{star}(v)=\{e\in E(G)\mid \partial_0 e=v\}.
	\end{equation*}
	A graph morphism $\alpha:G_1\to G_2$ is a \textit{covering} if it is surjective and the induced maps star$(v)\to$ star$(\alpha(v))$ are bijections. In this case we say that $G_1$ is a \textit{cover} of $G_2$.
\end{defn}
\bigskip

\begin{defn}\label{groupoids}(Groupoid)\\
	A \textit{groupoid} $\calG$ is a small category in which all morphisms are invertible. We will use Ob$(\calG)$ to denote the set of objects, and when referring to a morphism $g$ we will simply write $g\in\calG$. For $g\in \calG$, we will denote the initial and terminal objects by $i(g)$ and $t(g)$. For $x\in$ Ob$(\calG)$ we write $1_x$ for the identity morphism of $x$.\par 
	For $x,y\in$ Ob$(\calG)$ it will helpful to have the following additional notation.
	\begin{align*}
	\calG(x,y)&:=\{g\in\calG\mid i(g)=x,\,t(g)=y\}\\
	\calG(x,-)&:=\{g\in\calG\mid i(g)=x\}\\
	\calG(-,y)&:=\{g\in\calG\mid t(g)=y\}
	\end{align*}
	A \textit{subgroupoid} of $\calG$ is a subcategory in which all morphisms are invertible.
\end{defn}

When piecing together the finite cover in our proof of Leighton's Theorem, we will make use of groupoid actions. These are a direct analogue to group actions, and they also give rise to notions of orbit and stabiliser. The definition of groupoid action given below is from \cite[III.$\mathcal{G}$.2.8(3)]{nonpos}.

\begin{defn}(Groupoid action)\label{action}\\
	An \textit{action} of a groupoid $\calG$ on a set $A$ consists of a map $\varepsilon:A\to$ Ob$(\calG)$ and a map
		\begin{align*}
\{(g,a)\in\calG\times A\mid i(g)=\varepsilon(a)\}&\to A\\
(g,a)&\mapsto g\cdot a,
		\end{align*} 
		such that 
		\begin{enumerate}[(a)]
			\item $\varepsilon(g\cdot a)=t(g)$,
			\item $(g'g)\cdot a=g'\cdot(g\cdot a)$,
			\item $1_{\varepsilon(a)}\cdot a=a$,
	\end{enumerate}
for any $a\in A$ and $g,g'\in \calG$ satisfying $i(g)=\varepsilon(a)$ and $i(g')=t(g)$.\par 
\end{defn}

\begin{defn}(Orbits and stabilisers of groupoid actions)\\
	If a groupoid $\calG$ acts on a set $A$ and $\mathcal{H}\subset\calG$, define the \textit{$\mathcal{H}$-orbit of $a\in A$} by
	\begin{equation*}
	\mathcal{H}\cdot a:=\{h\cdot a\mid h\in\mathcal{H},\,i(h)=\varepsilon(a)\}.
	\end{equation*}
	Similarly, for $B\subset A$ write $\mathcal{H}\cdot B:=\cup_{b\in B}\mathcal{H}\cdot b$. Define the \textit{stabiliser of $a\in A$} by
	\begin{equation*}
	\text{Stab}_\calG(a):=\{g\in\calG\mid i(g)=\varepsilon(a),\,g\cdot a=a\}.
	\end{equation*}
\end{defn}

\bigskip
When building the finite cover in Leighton's Theorem we must find appropriate matchings between the pieces we wish to stitch together. The following lemma will help us achieve this.

\begin{lem}(Groupoid Orbit-Stabiliser Theorem)\label{orbstab}\\
	Let $\calG$ be a groupoid acting on a set $A$, and fix $a\in A$ with $\varepsilon(a)=x$. Then the fibres of the map
	\begin{align*}
	\phi_a:\calG(x,-)&\to\calG\cdot a\\
	g&\mapsto g\cdot a
	\end{align*}
	are cosets $g\Stab_\calG(a):=\{gg'\,|\,g'\in\Stab_\calG(a)\}$ for $g\in\calG(x,-)$. If $\calG(x,-)$ is finite, we deduce that
	\begin{equation*}
	|\calG(x,-)|=|\Stab_\calG(a)||\calG\cdot a|.
	\end{equation*}	
	\end{lem}
\begin{proof}
	If $g,h\in\calG(x,-)$ satisfy $g\cdot a=h\cdot a$, then $g^{-1}h\in$ Stab$_\calG(a)$, so $h\in g\text{Stab}_\calG(a)$.
\end{proof}

\bigskip
\section{Original Leighton's Theorem}\label{sec:Original}

We now present our new groupoid proof of Leighton's Theorem.

\begin{thm}(Leighton's Theorem)\label{Leighton}\\
Let $G_1$ and $G_2$ be finite connected graphs with a common cover. Then they have a common finite cover.	
\end{thm}
\begin{proof}
	We will define a finite groupoid $\calS$, consisting of maps between stars in $G_1$ and $G_2$, and we'll use this to label the vertices of a finite graph $G$. We'll then consider an action of $\calS$ on the edges of $G_1$ and $G_2$, and use this to connect up the vertices of $G$ with edges in a way that makes it a cover of $G_1$ and $G_2$.\par 
	We divide the proof into four steps. We define $\calS$ in the first step, then set up the action of $\calS$ in the second step. In the third step we build the graph $G$, using the action to connect up the vertices, and finally we construct the covering maps to $G_1$ and $G_2$ in the fourth step.
	\begin{enumerate}[Step 1:]
		
\item	We will have Ob$(\calS)= V(G_1)\sqcup V(G_2)$, and each $s\in\calS$ will be a bijection
	\begin{equation}
s:\text{star}(i(s))\to\text{star}(t(s)).
	\end{equation}
	 Composition of groupoid elements is just composition of bijections. The set of all such bijections forms a groupoid. Let $p_1:T\to G_1$ and $p_2:T\to G_2$ be coverings of $G_1$ and $G_2$ by a tree $T$. We define $\calS$ to be the subgroupoid consisting of bijections
	\begin{equation}\label{starmap}
	s:=p_i\circ g\circ(p_j|_{\text{star}(x)})^{-1}:\text{star}(p_j(x))\to\text{star}(p_ig(x)),
	\end{equation}
	for $i,j\in\{1,2\}$, $x\in V(T)$ and $g\in\Aut(T)$.
	Think of $s$ as a map that lifts a star in $G_1$ to a star in $T$, maps it across to another star in $T$, and then projects it down to a star in $G_2$. $\calS$ is closed under composition because if $x,y\in V(T)$ with $p_j(x)=p_j(y)$ then 
	$$(p_j|_{\text{star}(x)})^{-1}\circ p_j|_{\star*(y)}=g|_{\star*(y)},$$
	for some $g\in\Aut(T)$ a deck transformation of $p_j:T\to G_j$. $\calS$ also contains inverses because $\Aut(T)$ contains inverses.
	
\item	There is a natural action of $\calS$ on $E(G_1)\sqcup E(G_2)$ defined by $s\cdot e=s(e)$ for $\partial_0e=i(s)$ - the associated map $\varepsilon:E(G_1)\sqcup E(G_2)\to V(G_1)\sqcup V(G_2)$ is given by $\varepsilon(e):=\partial_0e$.  The following claim will be vital in the next step of the proof.\\
	\begin{claim}
		$\calS\cdot\overline{e}=\overline{\calS\cdot e}$
	\end{claim}\\
\begin{claimproof}
	It suffices to show the inclusion $\overline{\calS\cdot e}\subset\calS\cdot\overline{e}$, because replacing $e$ by $\overline{e}$ gives $\overline{\calS\cdot\overline{e}}\subset\calS\cdot e$, which implies $\calS\cdot\overline{e}\subset\overline{\calS\cdot e}$.
	So let $s$ be as in (\ref{starmap}) and take $\partial_0 e=i(s)=p_j(x)$. We will show there is $s'\in\calS$ with $s'\cdot\overline{e}=\overline{s\cdot e}$. Let $\hat{e}\in\star*(x)$ with $p_j(\hat{e})=e$, and consider $y:=\partial_1\hat{e}$. Put $f:=p_ig(\hat{e})=s\cdot e$, and define $s'\in\calS$ by
	$$s'=p_i\circ g\circ(p_j|{\star*(y)})^{-1}.$$
		 We have $i(s')=p_j(y)=\partial_1e$, so $s'\cdot \overline{e}$ is defined; and $s'\cdot \overline{e}=p_ig(\overline{\hat{e}})=\overline{f}=\overline{s\cdot e}$.
		 \end{claimproof} 
 \begin{figure}[H]
 	\definecolor{RED}{HTML}{FF0000}
 	\definecolor{BLUE}{HTML}{0000FF}
 	\centering
 	\begin{overpic}[width=.7\textwidth,clip=true]{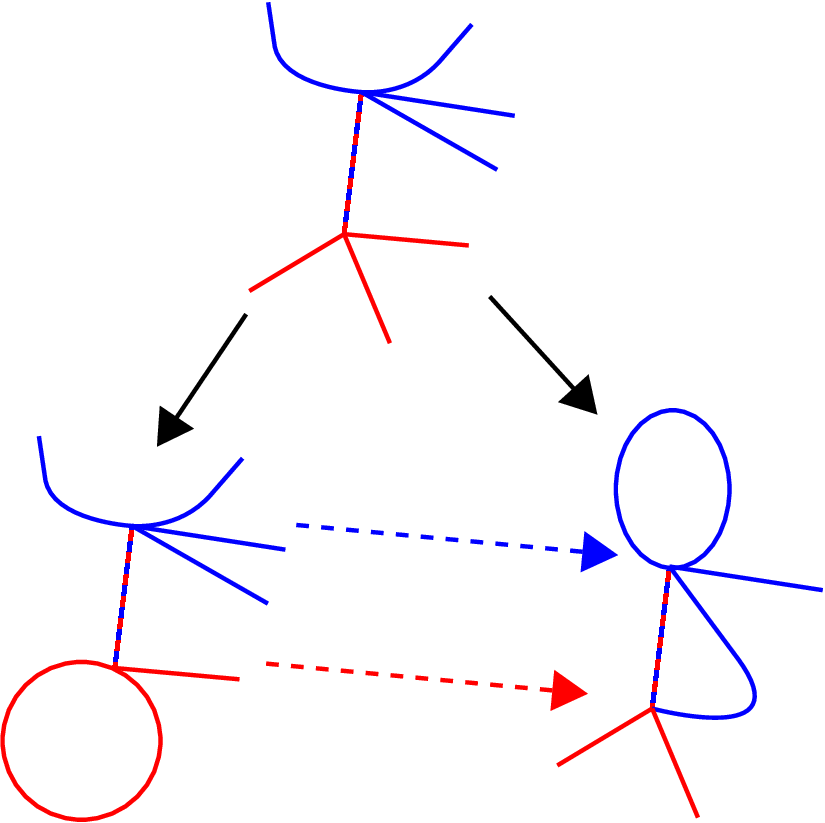}
 		\put(42,92){\Huge\textcolor{BLUE}{$y$}}
 		\put(50,38){\Huge\textcolor{BLUE}{$s'$}}
 		\put(38,62){\Huge\textcolor{RED}{$x$}}
 		\put(50,12){\Huge\textcolor{RED}{$s$}}
 		\put(10,90){\Huge$T$}
 		\put(-8,40){\Huge$G_j$}
 		\put(96,40){\Huge$G_i$}
 		\put(37,78){\Huge$\hat{e}$}
 		\put(10,26){\Huge$e$}
 		\put(74,20){\Huge$f$}
 		\put(18,60){\Huge$p_j$}
 		\put(68,60){\Huge$p_ig$}	
 	\end{overpic}
 	\caption{Diagram of $s$ and $s'$.}\label{fig:sands'}
 \end{figure}

	\item We are now ready to define a finite graph $G$ that covers $G_1$ and $G_2$. 
	Each vertex of $G$ will be labelled by a morphism in $\calS$, and we want vertices that map to a vertex $x\in V(G_1)$ to be labelled by a morphism in $\calS(x,-)$. If the covering $G\to G_1$ has degree $N$ then there must be $N$ vertices in $G$ sitting over each vertex of $G_1$. 
	The sets $\calS(x,-)$ could have different sizes for different $x$, so in general we need each morphism in $\calS$ to label multiple vertices in $G$.
	This motivates the following definition for the vertex set of $G$:
	\begin{equation*}
	V(G):=\left\{(s,l)\mid s\in\calS(x,y),\,x\in V(G_1),\,y\in V(G_2),\,1\leq l\leq\frac{N}{|\calS(x,-)|}\right\},
	\end{equation*}
	In a similar manner we define the edge set by
		\begin{equation*}
	E(G):=\left\{(e,f,k)\mid f\in\calS\cdot e,\,e\in E(G_1),\,f\in E(G_2),\,1\leq k\leq\frac{N}{|\calS\cdot e|}\right\}.
	\end{equation*}
	In the above two equations we take $N$ to be a fixed positive integer that is a common multiple of all the integers $|\calS(x,-)|$ and $|\calS\cdot e|$.
	We define edge inversion in $G$ by $\overline{(e,f,k)}:=(\overline{e},\overline{f},k)$, and this is well-defined by step 2.\par 
	\bigskip
	 To define $G$, it remains to define the map $\partial_0:E(G)\to V(G)$.	Fix $e\in E(G_1)$ and $f\in E(G_2)$ such that $f\in\calS\cdot e$, and say $\partial_0e=x$. For reasons that will become clear in the next step, each edge $(e,f,k)\in E(G)$ must satisfy
	\begin{align}\label{Gedgemap}
	\partial_0(e,f,k)&=(s,l),&\text{for some $s\in\calS(x,-)$ such that $s\cdot e=f$.}
	\end{align}
	We define such a map $\tau$ by choosing an arbitrary matching between such vertices $(s,l)$ and integers $1\leq k\leq N/|\calS\cdot e|$ - to verify that this is valid, we must check that we have equal numbers of each. Indeed, Lemma \ref{orbstab} tells us that $$|\{s\in\calS(x,-)\mid\calS\cdot e=f\}|=|\text{Stab}_\calS(e)|,$$
	 and so
	\begin{align*}
\left|\left\{(s,l)\mid s\in\calS(x,-),\,s\cdot e=f,\,1\leq l\leq\frac{N}{|\calS(x,-)|}\right\}\right|&=|\text{Stab}_\calS(e)|\frac{N}{|\calS(x,-)|}\\
&=\frac{N}{|\calS\cdot e|}
	\end{align*}
again by Lemma \ref{orbstab}.
	\bigskip
	
	\item In this last step we define the covering maps from $G$ down to $G_1$ and $G_2$.
	\begin{equation*}
	\begin{tikzcd}[
	ar symbol/.style = {draw=none,"#1" description,sloped},
	isomorphic/.style = {ar symbol={\cong}},
	equals/.style = {ar symbol={=}},
	subset/.style = {ar symbol={\subset}}
	]
	&G\ar{dl}[swap]{\mu_1}\ar{dr}{\mu_2}\\
	G_1&&G_2
	\end{tikzcd}
	\end{equation*}
	
	These are defined on the edge and vertex sets by
	\begin{align*}
	\mu_1(s,l):=i(s),&&\mu_1(e,f,k)&:=e,\\
	\mu_2(s,l):=t(s),&&\mu_2(e,f,k)&:=f.
	\end{align*}
	
	These maps clearly preserve edge inversion, and (\ref{Gedgemap}) ensures that they are well-defined graph morphisms, because then $\partial_0(e,f,k)=(s,l)$ with $i(s)=x$ implies
	\begin{align*}
	\partial_0\mu_1(e,f,k)&=\partial_0e&\partial_0\mu_2(e,f,k)&=\partial_0f\\
	&=x&&=\partial_0(s\cdot e)\\
	&=i(s)&&=t(s)\\
	&=\mu_1(s,l)&&=\mu_2(s,l)\\
	&=\mu_1\partial_0(e,f,k),&&=\mu_2\partial_0(e,f,k).
	\end{align*}
	By construction, the star of a vertex $(s,l)$ in $G$, with $s\in\calS(x,y)$, takes the form
	\begin{equation*}
	\text{star}(s,l)=\{(e,s\cdot e,k_e)\mid e\in\text{star}(x)\}
	\end{equation*}
	where each $k_e$ is some integer associated to $e$. Now $\mu_1(e,s\cdot e,k_e)=e$ and $\mu_2(e,s\cdot e,k_e)=s\cdot e$, so $\mu_1$ and $\mu_2$ induce bijections from $\text{star}(s,l)$ to $\text{star}(x)$ and $\text{star}(y)$ respectively. We conclude that $\mu_1$ and $\mu_2$ are coverings, which completes the proof of the theorem.
	\end{enumerate}
\end{proof}

\begin{remk}
	The finite cover $G$ constructed in the proof above may not be connected, but of course we can obtain a connected cover by choosing a component of $G$.
\end{remk}
\begin{remk}
	If we work with coloured graphs, meaning that each vertex and edge is assigned a colour, and we require graph morphisms to preserve colours, then Leighton's Theorem still holds with essentially the same proof. This is because we can define colours on the vertices and edges of $G$ by pulling back the colours from $G_1$, specifically $(s,l)$ will get the same colour as $i(s)$ and $(e,f,k)$ will get the same colour as $e$. $\Aut(T)$ will now be the group of colour preserving automorphisms and so each $s\in\calS$ will preserve colours, this implies that $i(s)$ has the same colour as $t(s)$ and each $(e,f,k)\in E(G)$ will have the same colour as $f$, therefore the covering $G\to G_2$ will also preserve colours.
\end{remk}
\begin{remk}
If there are coverings $\lambda_1:\hat{G}\to G_1$ and $\lambda_2:\hat{G}\to G_2$ such that $\hat{G}$ is connected and has a finite core $C$ (the \textit{core} of $\hat{G}$ is a subgraph $C$, minimal with respect to inclusion, such that the induced inclusion of topological realisations $|C|\xhookrightarrow{}|\hat{G}|$ is a homotopy equivalence - $C$ is unique if  $\hat{G}$ is not a tree), then we can arrange for there to be a covering $\tau:\hat{G}\to G$ of the finite graph $G$ that covers $G_1$ and $G_2$. Moreover we can do this so that the following diagram commutes.
\begin{equation}\label{corecommute}
\begin{tikzcd}[
ar symbol/.style = {draw=none,"#1" description,sloped},
isomorphic/.style = {ar symbol={\cong}},
equals/.style = {ar symbol={=}},
subset/.style = {ar symbol={\subset}}
]
&C\ar[bend right=30]{ddl}[swap]{\lambda_1}\ar[bend left=30]{ddr}{\lambda_2}\ar{d}{\tau}\\
&G\ar{dl}{\mu_1}\ar{dr}[swap]{\mu_2}\\
G_1&&G_2
\end{tikzcd}
\end{equation}
In general one cannot replace $C$ with $\hat{G}$ in this diagram. The way to construct such a $G$ is to let $q:T\to\hat{G}$ be a universal cover and define covering maps $p_1:=\lambda_1q:T\to G_1$ and $p_2:=\lambda_2q:T\to G_2$. For each $v\in V(C)$ let $z(v)$ be a lift to $T$, and let $s_{z(v)}\in\calS$ be defined by
$$s_{z(v)}:=p_2|_{\text{star}(z(v))}(p_1|_{\text{star}(z(v))})^{-1}.$$
 Then run the proof of the theorem as normal, but when it comes to piecing together $G$ with the edge map $\partial_0:E(G)\to V(G)$, do it so that we have an embedding $C\xhookrightarrow{}G$ given by $v\in V(C)\mapsto(s_{z(v)},l_v)$ and $d\in E(C)\mapsto(\lambda_1(d),\lambda_2(d),k_d)$ for $d\in E(C)$ and some choices of $l_v$ and $k_d$ (increasing the value of $N$ if needed). Note that this will satisfy (\ref{Gedgemap}) because if $\partial_0d=v$ then $\partial_0\lambda_1(d)=\lambda_1(v)=p_1(z(v))$ and
 \begin{align*}
s_{z(v)}\lambda_1(d)&=p_2|_{\text{star}(z(v))}(p_1|_{\text{star}(z(v))})^{-1}(\lambda_1(d))\\
&=\lambda_2q(\lambda_1q|_{\text{star}(z(v))})^{-1}(\lambda_1(d))\\
&=\lambda_2(d).
\end{align*}
By construction, the embedding $C\xhookrightarrow{}G$ will make (\ref{corecommute}) commute. The rest of $\hat{G}$ comprises subtrees $Y$, each with one vertex $v$ in $C$, so we can extend $C\xhookrightarrow{}G$ to a covering $\tau:\hat{G}\to G$ by lifting each map $\lambda_1:Y\to G_1$ up to $G$, with $v$ going to $(s_{z(v)},l_v)$.
\end{remk}

\bigskip
\section{Graph of objects version}\label{sec:GraphofObjects}

In this section we generalise Leighton's Theorem to graphs of objects, which concerns certain ``rigid'' types of coverings of graph of spaces.

One limitation to generalising Leighton's Theorem to graphs of spaces is illustrated by the following example.

\begin{exmp}\label{exmp:Baumslag}
	The Baumslag Solitar groups BS(1,3) and BS(2,2) both arise as fundamental groups of graphs of spaces with a single circular vertex space and a single circular edge space. For BS(1,3) the maps from edge space to vertex space will be coverings of degree 1 and 3, whereas for BS(2,2) they will both be coverings of degree 2. In both cases the sum of these degrees equals 4, and so both graphs of spaces are covered by a 4-regular tree of spaces in which all edge and vertex spaces are copies of the real line and all edge maps are homeomorphisms. However, there is no common finite cover for these graphs of spaces because BS(1,3) and BS(2,2) are not commensurable (BS(1,3) is solvable whereas BS(2,2) contains a non-abelian free subgroup).
\end{exmp}

This example would no longer work if we endowed the vertex and edge spaces with metrics and required the edge maps and coverings to respect these metrics, as then the tree of spaces would have two very different metrics induced by BS(1,3) and BS(2,2). This shows that the category of spaces we use to define our graphs of spaces matters. However, there are other more restrictive categories of spaces where Leighton's Theorem still fails. For example, in the category of graphs, Wise constructs a finite graph of spaces with non-residually finite fundamental group that is covered by a 4-regular tree of spaces in which all edge and vertex spaces are 6-regular trees and edge maps are isomorphisms \cite{Wise}. Such a tree of spaces also covers a finite graph of spaces whose fundamental group is a product of free groups, hence not commensurable to the group that Wise constructed.

It is clear then that one must impose some strong conditions in order to get a version of Leighton's Theorem for graphs of spaces. We do this by working with graphs of spaces with respect to a given category of spaces, and by restricting to coverings between graphs of spaces that induce isomorphisms between vertex spaces rather than just coverings. To emphasise that this is not the general setting we call them \emph{graphs of objects} instead of graphs of spaces. The full definitions of graphs of objects and their coverings are given below. Note that these definitions actually work with any category, not just categories of spaces - see Example \ref{exmp:spacecategories} for some suggestions of categories that could be used.

\begin{defn}(Graph of Objects)\label{graphobj}\\
	Let $\calC$ be a category and let $\calM_1\subset\calM_2\subset$ Hom$(\calC)$ be such that $($Ob$(\calC),\calM_1)$ and $($Ob$(\calC),\calM_2)$ are subcategories. Suppose that $($Ob$(\calC),\calM_1)$ is a groupoid. 
	
  A \textit{graph of objects} $X$ with respect to $(\calC,\calM_1,\calM_2)$ consists of
\begin{itemize}
\item a graph $G=G_X$,
\item objects $X_v\in$ Ob$(\calC)$ for $v\in V(G)$, called \textit{vertex objects},
\item objects $X_e\in$ Ob$(\calC)$ for $e\in E(G)$, called \textit{edge objects}, with $X_e=X_{\overline{e}}$,
\item and morphisms in $\calM_2$ 
\begin{align*}
\phi^e_0:X_e\to X_{\partial_0 e}\\
\phi^e_1:X_e\to X_{\partial_1 e},
\end{align*}
for $e\in E(G)$, called \textit{edge morphisms}, such that $\phi^e_0=\phi^{\overline{e}}_1$ and $\phi^e_1=\phi^{\overline{e}}_0$.
\end{itemize}
We say that a graph of objects $X$ is \textit{finite} if $G_X$ is finite.
\end{defn}

\begin{defn}(Morphisms between graphs of objects)\label{morgraphobj}\\
	Let $(\calC,\calM_1,\calM_2)$ be as above.
A \textit{morphism} $f:X\to Y$ between graphs of objects with respect to $(\calC,\calM_1,\calM_2)$ consists of
\begin{itemize}
\item a graph morphism $\hat{f}:G_X\to G_Y$,
\item morphisms $f_v:X_v\to X_{\hat{f}(v)}$ in $\calM_1$ for $v\in V(G_X)$,
\item and morphisms $f_e:X_e\to X_{\hat{f}(e)}$ in $\calM_1$ for $e\in E(G_X)$,
\end{itemize}
such that $f_e=f_{\overline{e}}$ and
\begin{equation}\label{2squares}
  \begin{tikzcd}[
  ar symbol/.style = {draw=none,"#1" description,sloped},
  isomorphic/.style = {ar symbol={\cong}},
  equals/.style = {ar symbol={=}},
  subset/.style = {ar symbol={\subset}}
  ]
 X_e\ar{d}{\phi^e_0}\ar{r}{f_e}& Y_{\hat{f}(e)}\ar{d}{\phi^{\hat{f}(e)}_0}\\
 X_u\ar{r}{f_u}&Y_{\hat{f}(u)}
  \end{tikzcd}\hspace{3cm}
    \begin{tikzcd}[
  ar symbol/.style = {draw=none,"#1" description,sloped},
  isomorphic/.style = {ar symbol={\cong}},
  equals/.style = {ar symbol={=}},
  subset/.style = {ar symbol={\subset}}
  ]
 X_e\ar{d}{\phi^e_1}\ar{r}{f_e}& Y_{\hat{f}(e)}\ar{d}{\phi^{\hat{f}(e)}_1}\\
 X_v\ar{r}{f_v}&Y_{\hat{f}(v)}
  \end{tikzcd}
\end{equation}

commute whenever $e\in E(G_X)$ is an edge from $u$ to $v$. Note that if the first square commutes for all edges $e$ then the second square will also commute for all edges as a consequence of the relations between the morphisms $\phi^e_0,\phi^e_1$ given in Definition \ref{graphobj}.
\end{defn}
\begin{defn}(Coverings of graphs of objects)\\
We say that a morphism $f:X\to Y$ between graphs of objects is a \textit{covering} if $\hat{f}$ is a covering of graphs. We say that $X$ is a \textit{cover} of $Y$. Similarly, we say that $f:X\to X$ is an \textit{automorphism} if $\hat{f}$ is a graph automorphism. Let Aut$(X)$ denote the group of automorphisms of $X$.
\end{defn}

\begin{exmp}\label{exmp:spacecategories}
	In the following table we give some examples of triples $(\calC,\calM_1,\calM_2)$ that can be used to define graphs of objects.\\
	\\
	\begin{tabular}{c||c|c|c}
		Name & $\calC$ & $\calM_1$ & $\calM_2$\\\hline
		\parbox[t]{3.5cm}{Graph of\\ topological spaces} & \parbox[t]{3.5cm}{topological spaces} & \parbox[t]{3.5cm}{homeomorphisms} & \parbox[t]{3.5cm}{continuous maps}\\[6mm]\hline
		\parbox[t]{3.5cm}{Graph of finite\\ simplicial complexes} & \parbox[t]{3.5cm}{finite simplicial\\ complexes} & \parbox[t]{3.5cm}{simplicial isomorphisms} & \parbox[t]{3.5cm}{simplicial maps}\\[6mm]\hline
		\parbox[t]{3.5cm}{Graph of finite\\ cube complexes} & \parbox[t]{3.5cm}{finite non-positively\\ curved cube complexes} & \parbox[t]{3.5cm}{cubical isomorphisms} & \parbox[t]{3.5cm}{locally isometric\\ cubical immersions}\\[6mm]\hline
		\parbox[t]{3.5cm}{Surface amalgams} & \parbox[t]{3.5cm}{compact surfaces\\ with boundary} & \parbox[t]{3.5cm}{homeomorphisms up\\ to isotopy} & \parbox[t]{3.5cm}{immersions up\\ to isotopy}\\[6mm]\hline
		\parbox[t]{3.5cm}{Graph of groups} & \parbox[t]{3.5cm}{groups} & \parbox[t]{3.5cm}{group isomorphisms} & \parbox[t]{3.5cm}{group monomorphisms}
	\end{tabular}\\
\\
Observe that the $\calM_1$ morphisms are all isomorphisms in the appropriate categories, as was required in Definition \ref{graphobj}.
Comparing with Definition \ref{graphobj}, we see that graphs of objects corresponding to the first four rows are just examples of graphs of spaces with respect to different categories. As discussed earlier, the difference in definitions comes when we consider morphisms and coverings. A morphism of graphs of spaces is usually defined in a similar way to Definition \ref{morgraphobj}, except that the morphisms $f_v$ and $f_e$ between vertex and edge spaces are not required to lie in $\calM_1$. Similarly, a covering of graphs of spaces usually requires that the maps $f_v$ and $f_e$ are coverings of vertex and edge spaces, but this is still weaker than having them in $\calM_1$, and the map $\hat{f}$ of underlying graphs is not required to be a covering of graphs. We can think of coverings of graphs of objects as ``rigid'' examples of coverings of graphs of spaces.
Note in particular that the coverings of graphs of spaces described in Example \ref{exmp:Baumslag} and the subsequent paragraph are not coverings of graphs of objects.

The last entry in the table, graphs of groups, is not an example of a graph of spaces, but it still gives a valid example of a graph of objects. However the usual notions of morphism and covering between graphs of groups are again weaker than those for graphs of objects, because we require the maps $f_v$ and $f_e$ to be group isomorphisms for graphs of objects.
\end{exmp}
\bigskip

To prove Leighton's Theorem for graphs of objects it turns out that we need a certain finiteness condition on the edge spaces, and so we'll need the following definition.

\begin{defn}(Isotropy groups)\\
	Let $X$ be a graph of objects and $H\leqslant$ Aut$(X)$. For each $e\in E(G_X)$ and $v\in V(G_X)$ define the \textit{isotropy groups} of $e$ and $v$ in $H$ as
	\begin{align*}
	H(e):=\{h_e\mid h\in H,\,\hat{h}(e)=e\}\leqslant\Aut_{\calM_1}(X_e),\\
	H(v):=\{h_v\mid h\in H,\,\hat{h}(v)=v\}\leqslant\Aut_{\calM_1}(X_v).
	\end{align*}
	These are different from the stabilisers $H_e$ and $H_v$; there are homomorphisms $H_e\to H(e)$ and $H_v\to H(v)$ sending $h$ to $h_e$ or $h_v$ respectively, these are surjective but not necessarily injective.
	
\end{defn}
\bigskip

We can now state our version of Leighton's Theorem for graphs of objects. See Example \ref{false} for why the assumption of finite edge isotropy groups is necessary.

\begin{thm}(Graph of Objects Leighton's Theorem)\\\label{thm:BasicObjectLeighton}
Let $X^1$ and $X^2$ be finite graphs of objects covered by a tree of objects $X$. If $\Aut(X)$ has finite edge isotropy groups, then $X^1$ and $X^2$ have a common finite cover.
\end{thm}

\begin{remk}
	The assumption that $\Aut(X)$ has finite edge isotropy groups will be automatically satisfied if the edge objects have finite $\calM_1$-automorphism groups. This is the case for graphs of finite simplicial or cube complexes from Example \ref{exmp:spacecategories}, and also for surface amalgams if we assume that the edge objects are annuli. For the latter example the vertex objects might have infinite $\calM_1$-automorphism groups (in this case each vertex object is a surface and the $\calM_1$-automorphism group is the mapping class group), but this is not a problem because Theorem \ref{thm:BasicObjectLeighton} does not require $\Aut(X)$ to have finite vertex isotropy groups.
\end{remk}
\bigskip

We will actually prove a stronger version of Theorem \ref{thm:BasicObjectLeighton} which incorporates a notion of symmetry-restricted closure analogous to that of Theorem \ref{thm:SymLeighton}. For this we need two more definitions.

\begin{defn}(Symmetry-restricted closure)\\\label{defn:ObSymClosure}
Let $X$ be a graph of objects and $H\leqslant$ Aut$(X)$. Define the \emph{symmetry-restricted closure of $H$} to be the subgroup $\mathscr{S}(H)\leqslant\Aut(X)$ consisting of automorphisms $g$ such that:
\begin{enumerate}[(1)]
	\item For all $e\in E(G_X)$ there exists $h\in H$ with $\hat{g}(e)=\hat{h}(e)$ and $g_e=h_e$.
	\item For all $v\in V(G_X)$ there exists $h\in H$ with $\hat{g}(v)=\hat{h}(v)$ and $g_v=h_v$.
\end{enumerate}
\end{defn}

\begin{defn}(Deck transformations)\\\label{fundgp}
	If $f:\tilde{X}\to X$ is a covering of graphs of objects such that $T:=G_{\tilde{X}}$ is a tree, then $\pi_1 G_X$ acts on $T$ as the group of deck transformations of the cover $\hat{f}:T\to G_X$. We also get an action of $\pi_1 G_X$ on $\tilde{X}$ defined by the homomorphism $\rho:\pi_1 G_X\to\Aut(\tilde{X})$, where $\widehat{\rho(g)}$ is given by the aforementioned action of $\pi_1 G_X$ on $T$, and the morphisms $\rho(g)_v,\rho(g)_e\in\calM_1$ are uniquely determined by the equations $f_v=f_{g(v)}\rho(g)_v$ and $f_e=f_{g(e)}\rho(g)_e$ (remember $\calM_1$ forms a groupoid). It is easy to check that $\rho(g)$ satisfies the commuting squares (\ref{2squares}) and that $\rho$ is a homomorphism. Clearly $f=f\rho(g)$ for all $g\in\pi_1 G_X$, and so we call the image of $\rho$ the \emph{group of deck transformations of the cover $f:\tilde{X}\to X$}.
\end{defn}
\bigskip

Our symmetry-restricted version of Leighton's Theorem for graphs of objects is the following.

\begin{thm}\label{thm:SymObjectLeighton}
	Let\begin{align*}
	f^1:\tilde{X}\to X^1\\
	f^2:\tilde{X}\to X^2
	\end{align*}
	be coverings of graphs of objects, with $G_1:=G_{X^1}$ and $G_2:=G_{X^2}$ both finite, and $T:=G_{\tilde{X}}$ a tree. Let $\Gamma_1$ and $\Gamma_2$ be the groups of deck transformations for $f^1$ and $f^2$, and suppose that $\Gamma_1,\Gamma_2\leqslant H\leqslant\Aut(\tilde{X})$.
	Suppose also that $H$ has finite edge isotropy groups. Then $X^1$ and $X^2$ have a common finite cover $X$, and there exists $g\in\mathscr{S}(H)$ that fits into the following commutative diagram of coverings.
	\begin{equation}\label{objdiagram}
	\begin{tikzcd}[
	ar symbol/.style = {draw=none,"#1" description,sloped},
	isomorphic/.style = {ar symbol={\cong}},
	equals/.style = {ar symbol={=}},
	subset/.style = {ar symbol={\subset}}
	]
	\tilde{X}\ar{dd}[swap]{f^1}\ar{rr}{g}\ar{dr}&&\tilde{X}\ar{dl}\ar{dd}{f^2}\\
	&X\ar{dl}{\mu^1}\ar{dr}[swap]{\mu^2}\\
	X^1&&X^2
	\end{tikzcd}
	\end{equation}
\end{thm}
\bigskip

This theorem can be stated more concisely by working entirely in the tree of objects $\tilde{X}$ as follows. 

\begin{thm}\label{thm:SymObjectLeighton2}
	Let $X$ be a tree of objects with $G_X=T$, and let $H\leqslant\Aut(X)$ be a subgroup with finite edge isotropy groups. Suppose that $\Gamma_1,\Gamma_2\leqslant H$ act freely cocompactly on $T$. Then there exists $g\in\mathscr{S}(H)$ such that $\Gamma_1^g$ is commensurable with $\Gamma_2$ in $\Aut(X)$.
\end{thm}
\bigskip

The equivalence of Theorems \ref{thm:SymObjectLeighton} and \ref{thm:SymObjectLeighton2} follows from the Galois correspondence for coverings of graphs of objects, which we describe in the following remark.

\begin{remk}\label{remk:GaloisObjGraph}
Let $\tilde{X}$ be a tree of objects with underlying tree $T$, and let $f:\tilde{X}\to X$ be a covering of a finite graph of objects with deck transformation group $\Gamma$. The usual covering theory of graphs gives us a correspondence between subgroups of $\Gamma$ and intermediate covers of $T\to G_X$; and given an intermediate cover of $T\to G_X$ there is a unique way of assigning edge and vertex objects (up to $\calM_1$-isomorphism) and morphisms that make diagrams (\ref{2squares}) commute, as we explain below. Thus we have a correspondence between subgroups of $\Gamma$ and intermediate covers of $f:\tilde{X}\to X$ (up to isomorphism).

An intermediate cover of graphs $T\overset{\hat{g}}{\to}G_Y\overset{\hat{h}}{\to}G_X$ induces an intermediate cover of graphs of objects $\tilde{X}\overset{g}{\to}Y\overset{h}{\to}X$ as follows (with $hg=f$). Define vertex and edge objects for $Y$ by $Y_u:=X_{\hat{h}(u)}$ and $Y_e:=X_{\hat{h}(e)}$, define the morphisms $h_u$ and $h_e$ by the identity, and define the edge morphisms in $Y$ so that diagram (\ref{2squares}) commutes for the map $h:Y\to X$. For  $u\in VT$ and $e\in ET$ we define $g_u:=h_{\hat{g}(u)}^{-1}f_u:\tilde{X}_u\to Y_{\hat{g}(u)}$ and $g_e:=h_{\hat{g}(e)}^{-1}f_e:\tilde{X}_e\to Y_{\hat{g}(e)}$, the map $g:\tilde{X}\to Y$ satisfies (\ref{2squares}) because $h$ and $f$ do. The uniqueness of this construction up to $\calM_1$-isomorphism essentially follows because each stage of the construction was forced by diagram (\ref{2squares}).
\end{remk}

Before proving Theorem \ref{thm:SymObjectLeighton}, we give an example to show that the assumption of finite edge isotropy groups is necessary.

\begin{exmp}\label{false}
	We work with graphs of topological spaces as in Example \ref{exmp:spacecategories}.
	We can exploit the infinite symmetry of the circle $S^1\subset\mathbb{C}$ to build finite graphs of objects $X^1, X^2$, with a common cover $\tilde{X}$, but no common finite cover.
	\begin{itemize}
		\item Let $X^1$ have a single vertex object $X^1_v$ and a single edge object $X^1_e$, both equal to $S^1$, and let the edge maps $\phi^e_0,\phi^e_1$ both be the identity.
		\item Let $X^2$ also have a single vertex object $X^2_v$ and a single edge object $X^2_e$, both equal to $S^1$, and let $\phi^e_0$ be the identity, but this time take $\phi^e_1$ to be the rotation $r:z\mapsto e^i z$. The important feature of this rotation is that it has infinite order in the homeomorphism group of $S^1$.
		\item Let $\tilde{X}$ have underlying graph consisting of an infinite chain of edges $(e_i)_{i\in\mathbb{Z}}$ and vertices $(v_i)_{i\in\mathbb{Z}}$ with $\partial_0 e_i=v_i$ and $\partial_1 e_i=v_{i+1}$. Let all the edge and vertex objects of $\tilde{X}$ equal $S^1$ and let all edge maps be the identity map.\par 
		There are covers
		\begin{align*}
		f^1:\tilde{X}\to X^1\\
		f^2:\tilde{X}\to X^2
		\end{align*}
		defined by
		\begin{enumerate}
			\item $\hat{f}^1(v_i)=v$, $\hat{f}^1(e_i)=e$, and $f^1_{v_i}=f^1_{e_i}=$ id$_{S^1}$ for all $i\in\mathbb{Z}$.
			\item $\hat{f}^2(v_i)=v$, $\hat{f}^2(e_i)=e$, and $f^2_{v_i}=f^2_{e_i}=r^i$ for all $i\in\mathbb{Z}$.
		\end{enumerate}
	\end{itemize}
	Why do $X^1$ and $X^2$ have no common finite cover? Well any finite cover $g^1:X\to X^1$ must be a circuit of copies of $S^1$, more precisely it must take the following form (up to isomorphism of $X$).
	\begin{itemize}
		\item $V(X)=\{v_1,...,v_n\}$ and $E(X)=\{e_1,...,e_n\}$ for some $n\in\mathbb{N}$.
		\item $\partial_0 e_i=v_i$ ($1\leq i\leq n$), $\partial_1 e_i=v_{i+1}$ ($1\leq i\leq n-1$) and $\partial_1 e_n=v_1$.
		\item $\hat{g}^1(v_i)=v$, $\hat{g}^1(e_i)=e$ for all $i$.
		\item $g^1_{v_i}=g^1_{e_i}=$ id$_{S^1}$ for all $i$.
	\end{itemize} 
	If there was a covering $g^2:X\to X^2$, there would be two possibilities for $\hat{g}^2$ corresponding to $\hat{g}^2(e_1)=e$ or $\overline{e}$. Suppose we're in the first case (the second will lead to a contradiction similarly), then $\hat{g}^2=\hat{g}^1$. Put $a=g^2_{v_1}$. The commutative squares (\ref{2squares}) then force $g^2_{e_1}=a, g^2_{v_2}=ra, g^2_{e_2}=ra, g^2_{v_3}=r^2a, g^2_{e_3}=r^2a$ and so on. But taking this right round the circuit we deduce that $g^2_{v_1}=r^na$, which is a contradiction because $r^n\neq$ id$_{S^1}$.
\end{exmp}

\bigskip
\begin{breakproof}[Proof of Theorem \ref{thm:SymObjectLeighton}]
	As for the original Leighton's Theorem, we will build a common finite cover by first constructing a finite groupoid $\calS$ consisting of ``maps between stars'' in $\tilde{X}$. But now each star is not just a set of edges meeting at a common vertex, as each star is endowed with the extra data of edge objects and edge morphisms. Thus these ``maps between stars'' in $\calS$ must have the additional data of morphisms between edge objects, and these morphisms must act naturally with respect to the edge morphisms. Once we have defined $\calS$, the proof will follow that of Theorem \ref{Leighton} quite closely - but with an extra step at the end to verify that we get a commutative diagram as in (\ref{objdiagram}).
\begin{enumerate}[Step 1:]
	\item Before constructing $\calS$, we will define a general notion of ``star map''.\par 
	Given graphs of objects $X,\,Y$ and $u\in V(G_X),\, v\in V(G_Y)$, a \textit{star map} from $u$ to $v$ is given by the data $s=(\hat{s},\,s_e\,:\,e\in\text{ star}(u))$, where:
	\begin{enumerate}
		\item $\hat{s}:\text{ star}(u)\to\text{ star}(v)$ is a bijection.
		\item $s_e:X_e\to Y_{\hat{s}(e)}$ is a morphism in $\calM_1$. There must also exist a morphism $s_u:X_u\to X_v$ in $\calM_1$ such that $s_u\phi^e_0=\phi^{\hat{s}(e)}_0s_e$ for all $e\in$ star$(u)$ - but $s_u$ is not part of the data of $s$ (for a given $s$ there could be many choices of $s_u$, whenever we write $s_u$ we refer to some arbitrary choice).
	\end{enumerate}
If $Z$ is another graph of objects and $w\in V(G_Z)$ and $t$ is a star map from $v$ to $w$, then we can compose $s$ with $t$ to produce a star map $ts$ from $u$ to $w$ with $\widehat{ts}=\hat{t}\hat{s}$, $(ts)_e=t_{\hat{s}(e)}s_e$ and $(ts)_u=t_vs_u$. There is a natural notion of identity star map at a vertex $u$ in which the morphisms $s_e$ will be identity morphisms, and any star map will have an inverse by replacing $\hat{s}$ and the morphisms $s_e$ with their inverses. Therefore the class of all star maps forms a category with inverses, where the objects are vertices in graphs of objects, and a star map $s$ from $u$ to $v$ has $i(s)=u$ and $t(s)=v$.\par 
If $f:X\to Y$ is a covering of graphs of objects, then for each $u\in V(G_X)$ there is a star map $f^u$ from $u$ to $\hat{f}(u)$ in which $\hat{f}^u$ is the restriction of $\hat{f}$ to star$(u)$, $f^u_e:=f_e$ and $f^u_u:=f_u$. And if $g:Y\to Z$ is another covering, then it is easy to check that $(gf)^u=g^{\hat{f}(u)}f^u$.\par 
Our groupoid $\calS$ will be the subcategory of star maps, with Ob$(\calS)=V(G_1)\sqcup V(G_2)$, and Hom$(\calS)$ consisting of star maps
\begin{equation}\label{starmap2}
s=(f^i)^{\hat{h}(z)}h^z((f^j)^z)^{-1}
\end{equation}
for $z\in V(T)$, $h\in H$ and $i,j\in\{1,2\}$. This is a star map from $u:=\hat{f}^j(z)$ to $v:=\hat{f}^i\hat{h}(z)$. Intuitively, think of $s$ as lifting the star of $u$ up to the star of $z$, mapping across to the star of $\hat{h}(z)$ by $h$, and then projecting down to the star of $v$ - a cartoon of this is given in Figure \ref{fig:starmap}.
 \begin{figure}[H]
 	\definecolor{RED}{HTML}{FF0000}
 	\definecolor{BLUE}{HTML}{0000FF}
 	\definecolor{ORANGE}{HTML}{FF6600}
 	\definecolor{PURPLE}{HTML}{800080}
 	\centering
 	\begin{overpic}[width=.7\textwidth,clip=true]{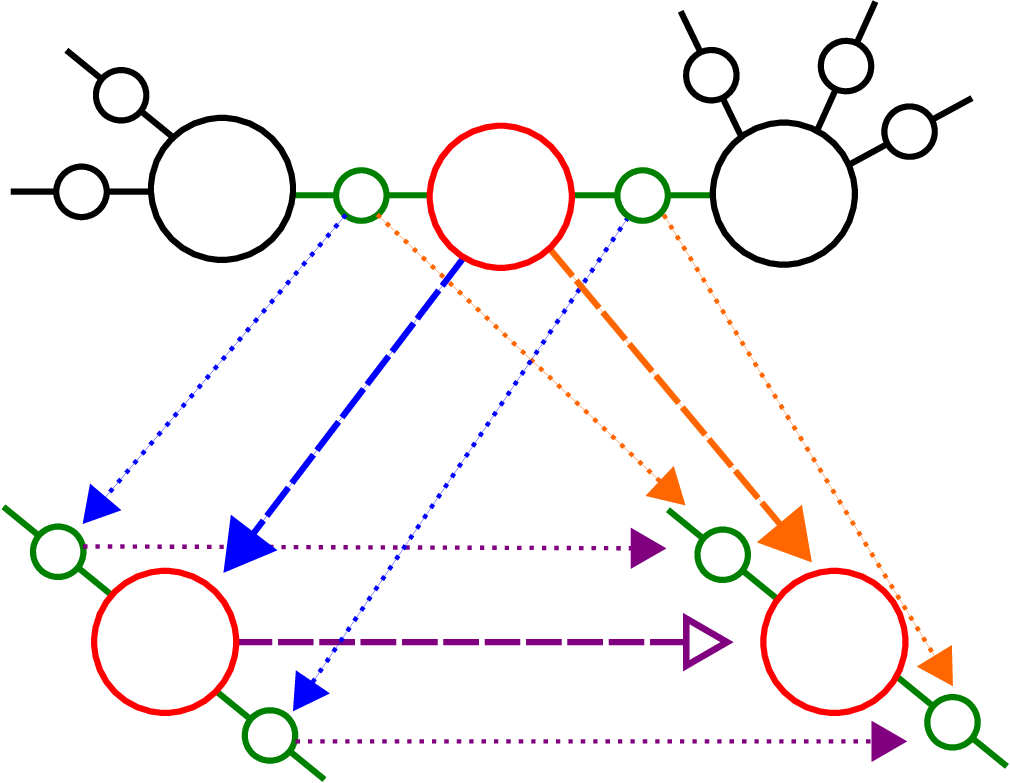}
 		\put(45,55){\Huge\textcolor{RED}{$\tilde{X}_z$}}
 		\put(5,40){\Huge\textcolor{BLUE}{$(f^j)^z$}}
 		\put(78,40){\Huge\textcolor{ORANGE}{$(f^i)^{\hat{h}(z)}h^z$}}
 		\put(50,8){\Huge\textcolor{PURPLE}{$s$}}	
 		\put(10,11){\Huge\textcolor{RED}{$X^j_u$}}
 		\put(78,11){\Huge\textcolor{RED}{$X^i_v$}}
 	\end{overpic}
	\caption{Diagram of the star map $s$ from (\ref{starmap2}).}\label{fig:starmap}
\end{figure}

 Unlike for the original Leighton's theorem, it is not obvious that $\calS$ is finite, so we prove this now.\\

\begin{claim}
$\calS$ is finite.
\end{claim}\\
\begin{claimproof}
It is enough to show that each $\calS(u,v)$ is finite. But $\calS(u,v)$ is a coset of $\calS(u,u)$, so it suffices to show that each $\calS(u,u)$ is finite. There is a homomorphism $\theta:\calS(u,u)\to$ Aut$($star$(u))$ with finite image, so it is enough to show that $\ker\theta$ is finite (Aut$($star$(u))$ is just the permutation group of the set star$(u)$).

Let $s\in\ker\theta$ and suppose $u\in V(G_1)$. Fix $z_0\in V(T)$ with $\hat{f}^1(z_0)=u$. By (\ref{starmap2}) we can write
\begin{equation*}
s=(f^1)^{z'}h^z((f^1)^z)^{-1}
\end{equation*}
for some $z,z'\in V(T)$ and $h\in H$ with $\hat{h}(z)=z'$. But then $\hat{f}^1(z)=\hat{f}^1(z')=u$, thus there exist $g^1,g^2\in\Gamma_1$ with $\hat{g}^1(z_0)=z$ and $\hat{g}^2(z')=z_0$, and we get
\begin{equation*}
s=(f^1)^{z_0}(g^2)^{z'}h^z(g^1)^{z_0}((f^1)^{z_0})^{-1}.
\end{equation*}
Moreover, $s\in\ker\theta$ so $\widehat{g^2hg^1}$ must fix star$(z_0)$, hence $(g^2hg^1)_e\in H(e)$ for each $e\in$ star$(z_0)$. By assumption the groups $H(e)$ are finite, hence there are only finitely many possibilities for $s\in\ker\theta$, as required.
\end{claimproof}
\bigskip

\item As for the original Leighton's Theorem, we now define an action of $\calS$ on a set of edge-related things, and show that it respects some notion of edge inversion. Define a finite set
\begin{equation*}
A:=\{(e,\hat{a}(e),a_e)\mid e\in E(G_1)\sqcup E(G_2),\,a\in\calS(\partial_0e,-)\}.
\end{equation*}
The observant will note that $A$ can also be given the structure of a groupoid, but we won't need this here.\par 
 We define an action of $\calS$ on $A$ by $s\cdot(e,\hat{a}(e),a_e):=(e,\widehat{s a}(e),(s a)_e)$ for $s\in\calS(t(a),-)$, with associated map $\varepsilon:A\to V(G_1)\sqcup V(G_2)$ given by $\varepsilon(e,\hat{a}(e),a_e)=\partial_0\hat{a}(e)=t(a)$. 
 
 We want to define an involution $A\to A$ given by $(e,\hat{a}(e),a_e)\mapsto(\overline{e},\overline{\hat{a}(e)},a_e)$. This is well-defined by the following claim.\\
 
 \begin{claim}
 	For any $(e,\hat{a}(e),a_e)\in A$ there exists $a'\in\calS(\partial_0\overline{e},-)$ with $\hat{a'}(\overline{e})=\overline{\hat{a}(e)}$ and $a'_{\overline{e}}=a_e$.
 \end{claim}\\
\begin{claimproof}
	As in (\ref{starmap2}), write
	$$a=(f^i)^{\hat{h}(z)}h^z((f^j)^z)^{-1}$$
	for $z\in V(T)$, $h\in H$ and $i,j\in\{1,2\}$, with $\hat{f}^j(z)=\partial_0e$. Let $\hat{e}\in\star*(z)$ with $\hat{f}^j(\hat{e})=e$ and put $z':=\partial_1\hat{e}$. Then define $a'\in\calS(\partial_0\overline{e},-)$ by
	$$a=(f^i)^{\hat{h}(z')}h^{z'}((f^j)^{z'})^{-1}.$$
	We have $i(a')=\hat{f}^j(z')=\hat{f}^j(\partial_1\hat{e})=\partial_1e=\partial_0\overline{e}$ as required. We also have $\hat{f}^j(\overline{\hat{e}})=\overline{e}$, so $\hat{a'}(\overline{e})=\hat{f}^i\hat{h}(\overline{\hat{e}})=\overline{\hat{f}^i\hat{h}(\hat{e})}=\overline{\hat{a}(e)}$. And finally we have
		\begin{align*}
		a'_{\overline{e}}&=(f^ih)_{\overline{\hat{e}}}(f^j_{\overline{e}})^{-1}\\
		&=(f^ih)_{\hat{e}}(f^j_e)^{-1}\\
		&=a_e.
		\end{align*}
\end{claimproof}\\ 
 
 $A$ can be partitioned into sets
 \begin{equation*}
 A(e):=\{(e,\hat{a}(e),a_e)\mid a\in\calS(\partial_0e,-)\},
 \end{equation*}
 for $e\in E(G_1)\sqcup E(G_2)$. These sets are related to the action of $\calS$ by the following claim.\\
\begin{claim}
$\calS\cdot(e,\hat{a}(e),a_e)=A(e)$
\end{claim}\\
\begin{claimproof}
The inclusion $\subset$ is clear from the definitions. The inclusion $\supset$ is also easy, because for $(e,\hat{s}(e),s_e)\in A(e)$ we have $s a^{-1}\cdot(e,\hat{a}(e),a_e)=	(e,\hat{s}(e),s_e)$.
\end{claimproof}
\bigskip

\item We can now construct our common finite cover $X$ of $X^1$ and $X^2$. The underlying graph $G:=G_X$ will have vertex set given by
\begin{equation*}
V(G):=\left\{(s,l)\mid s\in\calS(u_1,u_2),\,u_1\in V(G_1),\,u_2\in V(G_2),\,1\leq l\leq\frac{N}{|\calS(u_1,-)|}\right\},
\end{equation*}
and edge set given by
\begin{equation*}
E(G):=\left\{(e_1,e_2,m,k)\mid e_1\in E(G_1),\,e_2\in E(G_2),\,(e_1,e_2,m)\in A,\,1\leq k\leq\frac{N}{|A(e_1)|}\right\},
\end{equation*}
where $N$ is a fixed positive integer that is a common multiple of all the integers $|\calS(u_1,-)|$ and $|A(e_1)|$.\par
$A$ admits an involution $(e,\hat{a}(e),a_e)\mapsto(\overline{e},\overline{\hat{a}(e)},a_e)$, as in Step 2, which induces bijections $A(e)\to A(\overline{e})$. Hence we can define edge inversion in $G$ by $\overline{(e_1,e_2,m,k)}:=(\overline{e}_1,\overline{e}_2,m,k)$ (note that $a_{\overline{e}}=a_e$). \par 
\bigskip
Vertex and edge objects in $X$ will be given by
\begin{align*}
X_{(s,l)}:=X^1_{i(s)},&&X_{(e_1,e_2,m,k)}:=X^1_{e_1}.
\end{align*}

To complete the construction of $G$, we must define the map $\partial_0:E(G)\to V(G)$, and the edge morphisms in $X$. Fix $e_1\in E(G_1)$, $e_2\in E(G_2)$ and $(e_1,e_2,m)\in A$, and say $\partial_0e_q=u_q$ ($q=1,2$). We would like each edge $(e_1,e_2,m,k)\in E(G)$ to satisfy
\begin{align}\label{Wedgemap}
\partial_0(e_1,e_2,m,k)&=(s,l)
\end{align}
for some $s\in\calS(u_1,u_2)$ such that $\hat{s}(e_1)=e_2$ and $s_{e_1}=m$. Note this is equivalent to $s\cdot(e_1,e_1,1_{e_1})=(e_1,e_2,m)$, where $1_{e_1}$ is the identity morphism $X^1_{e_1}\to X^1_{e_1}$. We arrange this by choosing an arbitrary matching between such vertices $(s,l)$ and integers $1\leq k\leq N/|A(e_1)|$ - to verify that this is valid, we must check that we have equal numbers of each. Indeed, Lemma \ref{orbstab} tells us that
\begin{equation*}
|\{s\in\calS(u_1,-)\mid s\cdot(e_1,e_1,1_{e_1})=(e_1,e_2,m)\}|=|\text{Stab}_\calS(e_1,e_1,1_{e_1})|,
\end{equation*}  
and so
\begin{align*}
\left|\Set{(s,l)\,|\,\begin{aligned}s\in\calS(u_1,-),\,s\cdot(e_1,e_1,1_{e_1})=(e_1,e_2,m),\\1\leq l\leq\frac{N}{|\calS(u_1,-)|}\end{aligned}}\right|&=|\text{Stab}_\calS(e_1,e_1,1_{e_1})|\frac{N}{|\calS(u_1,-)|}\\
&=\frac{N}{|\calS\cdot(e_1,e_1,1_{e_1})|}\\
&=\frac{N}{|A(e_1)|}.
\end{align*}
The second equality is by Lemma \ref{orbstab} while the third equality follows from Step 2.

Finally, we define the edge morphisms in $X$ using the edge morphisms in $X^1$:
\begin{equation*}
\phi^{(e_1,e_2,m,k)}_0:=\phi^{e_1}_0:X^1_{e_1}\to X^1_{u_1}.
\end{equation*}
\bigskip

\item In this last step, we define coverings from $X$ down to $X^1$ and $X^2$.
\begin{equation*}
\begin{tikzcd}[
ar symbol/.style = {draw=none,"#1" description,sloped},
isomorphic/.style = {ar symbol={\cong}},
equals/.style = {ar symbol={=}},
subset/.style = {ar symbol={\subset}}
]
&X\ar{dl}[swap]{\mu^1}\ar{dr}{\mu^2}\\
X^1&&X^2.
\end{tikzcd}
\end{equation*}

 We define the maps $\hat{\mu}^1$ and $\hat{\mu}^2$ by
\begin{align*}
\hat{\mu}^1(s,l):=i(s),&&\hat{\mu}^1(e_1,e_2,m,k)&:=e_1,\\
\hat{\mu}^2(s,l):=t(s),&&\hat{\mu}^2(e_1,e_2,m,k)&:=e_2.
\end{align*}
These clearly preserve edge inversion, and $\partial_0(e_1,e_2,m,k)=(s,l)$ as in (\ref{Wedgemap}) implies
\begin{align*}
\partial_0\hat{\mu}^1(e_1,e_2,m,k)&=\partial_0e_1&\partial_0\hat{\mu}^2(e_1,e_2,m,k)&=\partial_0e_2\\
&=u_1&&=u_2\\
&=i(s)&&=t(s)\\
&=\hat{\mu}^1(s,l)&&=\hat{\mu}^2(s,l)\\
&=\hat{\mu}^1\partial_0(e_1,e_2,m,k),&&=\hat{\mu}^2\partial_0(e_1,e_2,m,k),
\end{align*}
thus $\hat{\mu}^1$ and $\hat{\mu}^2$ are well-defined graph morphisms.

We can then define $\mu^1$ and $\mu^2$ by the morphisms
\begin{align*}
\mu^1_{(s,l)}&:=1_{i(s)},&\mu^1_{(e_1,e_2,m,k)}&:=1_{e_1},\\
\mu^2_{(s,l)}&:=s_{i(s)},&\mu^2_{(e_1,e_2,m,k)}&:=m,
\end{align*}
where $1_{i(s)}$ is the identity morphism $X^1_{i(s)}\to X^1_{i(s)}$, and $s_{i(s)}$ is one of the possible vertex object morphisms associated to the star map $s$ (see part (b) in the definition of star map at the beginning of Step 1). Note that the above morphisms do go between the appropriate vertex and edge objects as specified by the maps $\hat{\mu}^1$ and $\hat{\mu}^2$.\par 
Again with $\partial_0(e_1,e_2,m,k)=(s,l)$ as in (\ref{Wedgemap}), we get the following commutative squares, demonstrating that $\mu^1$ and $\mu^2$ are well-defined morphisms of graphs of objects. The bottom left square commutes precisely because $s$ is a star map.
\begin{equation*}
\begin{tikzcd}[
ar symbol/.style = {draw=none,"#1" description,sloped},
isomorphic/.style = {ar symbol={\cong}},
equals/.style = {ar symbol={=}},
subset/.style = {ar symbol={\subset}}
]
X^1_{e_1}\ar{dd}{\phi^{e_1}_0}\ar[equals]{rr}&& X^1_{e_1}\ar{dd}{\phi^{e_1}_0}\\
\\
X^1_{u_1}\ar[equals]{rr}&&X^1_{u_1}
\end{tikzcd}
\hspace{2cm}
\text{\Huge$\Rightarrow$}
\hspace{2cm}
\begin{tikzcd}[
ar symbol/.style = {draw=none,"#1" description,sloped},
isomorphic/.style = {ar symbol={\cong}},
equals/.style = {ar symbol={=}},
subset/.style = {ar symbol={\subset}}
]
X_{(e_1,e_2,m,k)}\ar{dd}{\phi^{(e_1,e_2,m,k)}_0}\ar{rr}{\mu^1_{(e_1,e_2,m,k)}}&& X^1_{e_1}\ar{dd}{\phi^{e_1}_0}\\
\\
X_{(s,l)}\ar{rr}{\mu^1_{(s,l)}}&&X^1_{u_1}
\end{tikzcd}
\end{equation*}

\begin{equation*}
\begin{tikzcd}[
ar symbol/.style = {draw=none,"#1" description,sloped},
isomorphic/.style = {ar symbol={\cong}},
equals/.style = {ar symbol={=}},
subset/.style = {ar symbol={\subset}}
]
X^1_{e_1}\ar{dd}{\phi^{e_1}_0}\ar{rr}{m=s_{e_1}}&& X^2_{e_2}\ar{dd}{\phi^{e_2}_0}\\
\\
X^1_{u_1}\ar{rr}{s_{u_1}}&&X^2_{u_2}
\end{tikzcd}
\hspace{2cm}
\text{\Huge$\Rightarrow$}
\hspace{2cm}
\begin{tikzcd}[
ar symbol/.style = {draw=none,"#1" description,sloped},
isomorphic/.style = {ar symbol={\cong}},
equals/.style = {ar symbol={=}},
subset/.style = {ar symbol={\subset}}
]
X_{(e_1,e_2,m,k)}\ar{dd}{\phi^{(e_1,e_2,m,k)}_0}\ar{rr}{\mu^2_{(e_1,e_2,m,k)}}&& X^2_{e_2}\ar{dd}{\phi^{e_2}_0}\\
\\
X_{(s,l)}\ar{rr}{\mu^2_{(s,l)}}&&X^2_{u_2}
\end{tikzcd}
\end{equation*}

By construction, the star of a vertex $(s,l)$ in $G$, with $s\in\calS(u_1,u_2)$, takes the form
\begin{equation*}
\text{star}(s,l)=\{(e,\hat{s}(e),s_e,k_e)\mid e\in\text{star}(u_1)\}
\end{equation*}
where each $k_e$ is some integer associated to $e$. Now $\hat{\mu}^1(e,\hat{s}(e),s_e,k_e)=e$ and $\hat{\mu}^2(e,\hat{s}(e),s_e,k_e)=\hat{s}(e)$, so $\hat{\mu}^1$ and $\hat{\mu}^2$ induce bijections from $\text{star}(s,l)$ to $\text{star}(u_1)$ and $\text{star}(u_2)$ respectively. We conclude that $\hat{\mu}^1$ and $\hat{\mu}^2$ are graph coverings, which makes $\mu^1$ and $\mu^2$ coverings of graphs of objects.

\item Finally we must construct diagram (\ref{objdiagram}) with $g\in\mathscr{S}(H)$. We can assume that the graph $G$ is connected, as otherwise we just restrict to a component. The usual covering space theory of graphs allows us to draw the following commutative diagram of graph coverings. (More precisely, $\hat{\nu}^1$ is a lift of $\hat{f}^1$ with respect to $\hat{\mu}^1$, $\hat{\nu}^2$ is a lift of $\hat{f}^2$ with respect to $\hat{\mu}^2$, and $\hat{g}$ is a lift of $\hat{\nu}^1$ with respect to $\hat{\nu}^2$.)

\begin{equation}
\begin{tikzcd}[
ar symbol/.style = {draw=none,"#1" description,sloped},
isomorphic/.style = {ar symbol={\cong}},
equals/.style = {ar symbol={=}},
subset/.style = {ar symbol={\subset}}
]
T\ar{dd}[swap]{\hat{f}^1}\ar{rr}{\hat{g}}\ar{dr}[swap]{\hat{\nu}^1}&&T\ar{dl}{\hat{\nu}^2}\ar{dd}{\hat{f}^2}\\
&G\ar{dl}{\hat{\mu}^1}\ar{dr}[swap]{\hat{\mu}^2}\\
G_1&&G_2
\end{tikzcd}
\end{equation}

As in Remark \ref{remk:GaloisObjGraph}, there is then a unique way of upgrading $\hat{\nu}^1$, $\hat{\nu}^2$ and $\hat{g}$ to coverings of graphs of objects $\nu^1$, $\nu^2$ and $g$. So we now have the following commutative diagram of graphs of objects, with $g\in\Aut(\tilde{X})$.

\begin{equation}
\begin{tikzcd}[
ar symbol/.style = {draw=none,"#1" description,sloped},
isomorphic/.style = {ar symbol={\cong}},
equals/.style = {ar symbol={=}},
subset/.style = {ar symbol={\subset}}
]
\tilde{X}\ar{dd}[swap]{f^1}\ar{rr}{g}\ar{dr}[swap]{\nu^1}&&\tilde{X}\ar{dl}{\nu^2}\ar{dd}{f^2}\\
&X\ar{dl}{\mu^1}\ar{dr}[swap]{\mu^2}\\
X^1&&X^2
\end{tikzcd}
\end{equation}

It remains to prove that $g\in\mathscr{S}(H)$. Consider a star map $s\in\mathcal{S}$ with $i(s)=u_1\in V(G_1)$ and $t(s)=u_2\in V(G_2)$. Suppose that $s$ takes the form
$$s=(f^2)^{z_2}h^{z_1}((f^1)^{z_1})^{-1}$$
from (\ref{starmap2}), with $z_1,z_2\in V(T)$, $\hat{f}^i(z_i)=u_i$, $h\in H$, and $\hat{h}(z_1)=z_2$. We can then choose the morphism $s_{u_1}$ to fit into the following commutative diagram.

\begin{equation}
\begin{tikzcd}[
ar symbol/.style = {draw=none,"#1" description,sloped},
isomorphic/.style = {ar symbol={\cong}},
equals/.style = {ar symbol={=}},
subset/.style = {ar symbol={\subset}}
]
\tilde{X}_{z_1}\ar{d}{f^1_{z_1}}\ar{r}{h_{z_1}}&\tilde{X}_{z_2}\ar{d}{f^2_{z_2}}\\
X^1_{u_1}\ar{r}{s_{u_1}}&X^2_{u_2}
\end{tikzcd}
\end{equation}
If $X_{(s,l)}$ is a vertex object of $X$, and $v_1,v_2\in V(T)$ are such that $\hat{\nu}^i(v_i)=(s,l)$ (so $\hat{f}^i(v_i)=u_i$) and $\hat{g}(v_1)=v_2$, then we get a larger commutative diagram as follows.

\begin{equation}\label{bigdiagram}
\begin{tikzcd}[
ar symbol/.style = {draw=none,"#1" description,sloped},
isomorphic/.style = {ar symbol={\cong}},
equals/.style = {ar symbol={=}},
subset/.style = {ar symbol={\subset}}
]
\tilde{X}_{z_1}\ar[bend right=20]{dddr}[swap]{f^1_{z_1}}\ar{rrrr}{h_{z_1}}\ar{dr}[swap]{g^1_{z_1}}&&&&\tilde{X}_{z_2}\ar[bend left=20]{dddl}{f^2_{z_2}}\ar{dl}{g^2_{z_2}}\\
&\tilde{X}_{v_1}\ar{dd}[swap]{f^1_{v_1}}\ar{rr}{g_{v_1}}\ar{dr}[swap]{\nu^1_{v_1}}&&\tilde{X}_{v_2}\ar{dd}{f^2_{v_2}}\ar{dl}{\nu^2_{v_2}}\\
&&X_{(s,l)}\ar{dl}{\mu^1_{(s,l)}}\ar{dr}[swap]{\mu^2_{(s,l)}}\\
&X^1_{u_1}\ar{rr}[swap]{s_{u_1}}&&X^2_{u_2}
\end{tikzcd}
\end{equation}
Here $g^i\in\Gamma_i$ is the element of the deck transformation group with $\hat{g}^i(z_i)=v_i$. Since $\Gamma_1,\Gamma_2\leqslant H$, the top square of (\ref{bigdiagram}) implies that $g_{v_1}=h'_{v_1}$ for some $h'\in H$.

A very similar argument can be run for edge objects, and so we conclude that $g\in\mathscr{S}(H)$.
\end{enumerate}
\end{breakproof}

\begin{remk}\label{remk:basedg}
	Given $v\in V(T)$, we can choose the automorphism $g\in\Aut(\tilde{X})$ from Theorem \ref{thm:SymObjectLeighton} such that $\hat{g}(v)=v$ and $g_v=1_{\tilde{X}_v}$ (and similarly for $e\in E(T)$). This follows by examining Step 5 of the proof. Indeed, let $s\in\mathcal{S}$ be defined by $s=(f^2)^v((f^1)^v)^{-1}$, and when restricting to a component of $G$ at the beginning of Step 5 make sure that the vertex $(s,1)$ is included. Then choose the coverings $\hat{\nu}^1$, $\hat{\nu}^2$ and $\hat{g}$ so that $\hat{\nu}^i(v)=(s,1)$ and $\hat{g}(v)=v$. If $\hat{f}^i(v)=u_i$, then we can draw diagram (\ref{bigdiagram}) with $z_i=v_i=v$ and $g^i=h=1\in\Aut(\tilde{X})$. It follows that $g_v=1_{\tilde{X}_v}$ as required.
\end{remk}

\bigskip
\section{Symmetry-restricted version}\label{sec:SymVersion}

We now show how to deduce the Symmetry-restricted Leighton's Theorem from its graph of objects counterpart, Theorem \ref{thm:SymObjectLeighton2}.

\begin{thm}(Symmetry-restricted Leighton's Theorem)\\\label{thm:SymLeighton2}
	Let $T$ be a tree, and $H\leqslant\Aut(T)$, and let $\Gamma_1, \Gamma_2 \leqslant H$ be free uniform lattices in $\Aut(T)$.
	Then for all $R \in \mathbb{N}$ there exists $g \in \mathscr{S}_R(H)$ such that $\Gamma_1^g$ is commensurable to $\Gamma_2$ in $\Aut(T)$. 
\end{thm}
\begin{proof}
We will turn $T$ into a tree of objects $X$ (ie. $G_X=T$). This will be with respect to $(\mathcal{C},\mathcal{M}_1,\mathcal{M}_2)$, where $\calC$ is the category of pairs $(Y,U)$ for $Y$ a finite tree and $U\subset V(Y)$, and a morphism in $\calC$ from $(Y,U)$ to $(Y',U')$ is a tree embedding $Y\xhookrightarrow{} Y'$ such that $U'$ is contained in the image of $U$. A morphism is in $\calM_1$ if $Y\to Y'$ is an isomorphism and $U'$ equals the image of $U$, and all morphisms are in $\calM_2$. We then define the vertex objects for $X$ as based $R$-balls $X_v:=(B_R(v),\{v\})$ for $v\in V(T)$, and the edge objects as $X_e:=(N_{R-1}(e),\{\partial_0e,\partial_1e\})$ for $e\in E(T)$, where $N_{R-1}(e)$ is the $(R-1)$-neighbourhood of $e$. The morphisms $\phi_0^e:X_e\to X_{\partial_0e}$ are given by the inclusions $N_{R-1}(e)\xhookrightarrow{}B_R(\partial_0e)$.

For $g\in\Aut(T)$ let $g_v$ be the restriction of $g$ to $B_R(v)$ (as in Definition \ref{defn:SymClosure}) and let $g_e$ be the restriction of $g$ to $N_{R-1}(e)$. We then have a homomorphism $\psi:\Aut(T)\to\Aut(X)$ defined by $\widehat{\psi(g)}:=g$ and
\begin{align*}
\psi(g)_v&:=g_v:B_R(v)\to B_R(gv),\\
\psi(g)_e&:=g_e:N_{R-1}(e)\to N_{R-1}(ge).
\end{align*}

It is easy to check that $\psi$ is a well-defined homomorphism. The key point is that $\psi$ is actually an isomorphism, as we will now show.\\

\begin{claim}
	$\psi$ is an isomorphism.
\end{claim}\\
\begin{claimproof}
	The homomorphism $\psi$ admits a retraction $r:\Aut(X)\to\Aut(T)$ given by $g\mapsto\hat{g}$, so we must show that $r$ has trivial kernel. 
	
	Consider $g\in\ker(r)$ and pick $v\in V(T)$. 		
	Take a vertex $u\in B_R(v)$. We will show that $g_v(u)=u$ by induction on the distance $d(u,v)$. We know that $g_v(v)=v$ by definition of morphisms in $\calC$, so we may assume that $u\neq v$. Suppose that the segment $[v,u]$ in $T$ starts with the edge $e$ such that $\partial_0e=v$ and $\partial_1e=w$. As $g$ is a morphism of graphs of objects, we have the following commutative diagram.

	\begin{equation}\label{gvgw}
	\begin{tikzcd}[
	ar symbol/.style = {draw=none,"#1" description,sloped},
	isomorphic/.style = {ar symbol={\cong}},
	equals/.style = {ar symbol={=}},
	subset/.style = {ar symbol={\subset}}
	]
	B_R(v)\ar{d}{g_v}&N_{R-1}(e)\ar[hook',l]\ar[hook,r]\ar{d}{g_e}&B_R(w)\ar{d}{g_w}\\
	B_R(v)&N_{R-1}(e)\ar[hook',l]\ar[hook,r]&B_R(w)
	\end{tikzcd}
	\end{equation}
	We know that $d(u,w)<d(u,v)$, so $u\in B_R(w)$, and by induction we have $g_w(u)=u$. Diagram (\ref{gvgw}) then implies that $g_v$ also fixes $u$. This holds for all $u\in B_R(v)$, so $g_v$ is the identity map on $B_R(v)$, and diagram (\ref{gvgw}) implies that $g_e$ is the identity map on $N_{R-1}(e)$. Therefore $g$ is the identity on all edge and vertex objects, so $g=1\in\Aut(X)$ as required.
\end{claimproof}\\

As a consequence of the claim we know that $\hat{g}_v=g_v$ and $\hat{g}_e=g_e$ for any $g\in\Aut(X)$, $v\in V(T)$ and $e\in E(T)$. It follows easily that, for $H\leqslant\Aut(T)$, the $R$-symmetry-restricted closure from Definition \ref{defn:SymClosure} corresponds to the symmetry-restricted closure from Definition \ref{defn:ObSymClosure}:
$$\psi(\mathscr{S}_R(H))=\mathscr{S}(\psi(H))$$
We are then done by Theorem \ref{thm:SymObjectLeighton2}.
\end{proof}

\begin{remk}
	It follows from Remark \ref{remk:basedg} that given $v\in V(T)$ we can choose the conjugating automorphism $g\in\Aut(T)$ to restrict to the identity on the ball $B_R(v)$.
\end{remk}

Recall from Example \ref{exmp:graphoftriangles} that a \emph{graph of polygons} is a space consisting of solid regular polygons with some edges joining vertices of the polygons. And recall that a \emph{covering} of graphs of polygons is a topological covering that restricts to isometries between polygons. We now give two proofs of why Leighton's Theorem holds for graphs of polygons, the first proof views graphs of polygons as graphs of objects while the second proof uses the Symmetry-restricted Leighton's Theorem.

\begin{prop}\label{prop:polygonLeighton}
	Leighton's Theorem holds for graphs of polygons: if two finite graphs of polygons (ie. with finite underlying graphs) are covered by the same tree of polygons, then they have a common finite cover.
\end{prop}
\begin{proof}[Proof 1.]
	Graphs of polygons are graphs of objects with respect to
	 $$(\calC,\calM_1,\calM_2)=(\text{metric spaces, isometries, isometric embeddings}).$$
	  The vertex objects are polygons and the edge objects are points. The proposition then follows from Theorem \ref{thm:BasicObjectLeighton}.
\end{proof}
\begin{proof}[Proof 2.]
	Let $X$ be a graph of polygons. Each polygon contains a \emph{polygon star} consisting of the centre of the polygon and edges joining the centre to each vertex. Let $X^*$ be the graph obtained from $X$ by retracting each polygon to its polygon star. This induces an $\Aut(X)$-invariant retraction $X\to X^*$, which in turn induces an injective homomorphism $\Aut(X)\to\Aut(X^*)$. An example is illustrated in Figure \ref{fig:XtoX*}.\\
	
\begin{figure}[H]
	\centering
	\scalebox{0.7}{
		\begin{tikzpicture}[node distance = 4cm,
		scale=3,
		triangle/.style = {fill=red, draw=red, regular polygon, regular polygon sides=3, minimum size=40pt},
		square/.style = {fill=red, draw=red, regular polygon, regular polygon sides=4, minimum size=40pt},
		pentagon/.style = {fill=red, draw=red, regular polygon, regular polygon sides=5, minimum size=40pt},
		every loop/.style={min distance=2cm},
		hull/.style={draw=none}
		]
		
		\tikzstyle{label}=[draw=none,font=\Huge]
		
		\begin{scope}[shift={(-2,0)}]
		\node[hull] (X) at (2.5,-0.2) {};
		\node[label] (XL) at (2.5,-1) {$X$};
		\node[triangle] (1) {};
		\node[pentagon] (2) at (2,0) {};
		\node[square] (3) at (1,-1) {};
		
		\path
		(1.corner 3) edge (3.corner 2)
		(1.corner 2) edge [bend right]  (3.corner 2)
		(1.corner 1) edge (2.corner 1)
		(2.corner 3) edge [loop below] (2.corner 3)
		;
		\end{scope}
		
		\begin{scope}[shift={(2,0)}]
		\node[hull] (X*) at (-0.4,-0.2) {};
		\node[label] (X*L) at (2.5,-1) {$X^*$};
		\node[triangle, fill=none,draw=none] (1) {};
		\node[pentagon, fill=none,draw=none] (2) at (2,0) {};
		\node[square, fill=none,draw=none] (3) at (1,-1) {};
		
		\draw[red] (1.center)--(1.corner 1);
		\draw[red] (1.center)--(1.corner 2);
		\draw[red] (1.center)--(1.corner 3);
		\draw[red] (2.center)--(2.corner 1);
		\draw[red] (2.center)--(2.corner 2);
		\draw[red] (2.center)--(2.corner 3);
		\draw[red] (2.center)--(2.corner 4);
		\draw[red] (2.center)--(2.corner 5);
		\draw[red] (3.center)--(3.corner 1);
		\draw[red] (3.center)--(3.corner 2);
		\draw[red] (3.center)--(3.corner 3);
		\draw[red] (3.center)--(3.corner 4);
		
		\path
		(1.corner 3) edge (3.corner 2)
		(1.corner 2) edge [bend right]  (3.corner 2)
		(1.corner 1) edge (2.corner 1)
		(2.corner 3) edge [loop below] (2.corner 3)
		;
		\end{scope}
		
		\draw[draw=black,fill=blue,-triangle 90, ultra thick] (X) -- (X*);
		\end{tikzpicture}
	}
	\caption{The retraction $X\to X^*$.}\label{fig:XtoX*}
\end{figure}
	
	If two finite graphs of polygons $X_1$ and $X_2$ are covered by a tree of polygons $T$, with covering maps $p_i:T\to X_i$, then standard covering space theory tells us that any common finite cover $\hat{X}$ fits into a commutative diagram of covering maps as follows, where $g$ is an automorphism of the tree of polygons $T$.
	
	\begin{equation}\label{xhatdiagram2}
	\begin{tikzcd}[
	ar symbol/.style = {draw=none,"#1" description,sloped},
	isomorphic/.style = {ar symbol={\cong}},
	equals/.style = {ar symbol={=}},
	subset/.style = {ar symbol={\subset}}
	]
	T\ar{dd}[swap]{p_1}\ar{rr}{g}\ar{dr}[swap]{\nu_1}&&T\ar{dl}{\nu_2}\ar{dd}{p_2}\\
	&\hat{X}\ar{dl}{\mu_1}\ar{dr}[swap]{\mu_2}\\
	X_1&&X_2
	\end{tikzcd}
	\end{equation}
	Let the deck transformation group of $p_i:T\to X_i$ be denoted $\Gamma_i\leqslant\Aut(T)$. As described for graphs in the introduction, the existence of $\hat{X}$ is equivalent to the existence of $g\in\Aut(T)$ such that $\Gamma_1^g$ is commensurable to $\Gamma_2$, so let's now prove the latter.
	
	As explained above, we have a retraction $T\to T^*$ that induces an injective homomorphism $\Aut(T)\to\Aut(T^*)$. Let $H\leqslant\Aut(T^*)$ be the image. The edges in each polygon star have a cyclic ordering given by the cyclic ordering on the vertices of the polygon, and an automorphism of $T^*$ extends to an automorphism of $T$ if and only if it maps polygon stars to polygon stars in a way that preserves the cyclic orderings. In the notation of Definition \ref{defn:SymClosure}, this means that $H=\mathscr{S}_1(H)$. We may then apply Theorem \ref{thm:SymLeighton2} to obtain $g\in H$ that conjugates the image of $\Gamma_1$ in $H$ onto a subgroup commensurable with the image of $\Gamma_2$. Pulling this back to $\Aut(T)$ gives $g\in\Aut(T)$ such that $\Gamma_1^g$ is commensurable to $\Gamma_2$, as required.
\end{proof}
\bigskip

To close this section, we give an example showing that one cannot demand that the conjugating element $g$ lie in the subgroup $H$ in the Symmetry-restricted Leighton's Theorem (see also Remark \ref{notHbar}).
\begin{exmp}

	Consider a tree $T$ with directed red and blue edges as shown in Figure \ref{fig:xi}. Let $G_1$ be the rose on two petals, let $p_1:T\to G_1$ be the standard covering defined by the edge colouring of $T$, and let $\Gamma_1$ be the corresponding group of deck transformations. Let $\xi\in$ Aut$(T)$ reflect the right side of $T$ while fixing the left side, as shown above. Suppose $p_2:T\to G_2$ is the covering with deck transformation group $\Gamma_2=\xi^{-1}\Gamma_1 \xi$ (so $G_2$ is also the rose on two petals). Let $H$ be the subgroup generated by $\Gamma_1$ and $\Gamma_2$. We will see that no $h\in H$ makes $\Gamma_1^h$ and $\Gamma_2$ commensurable.
	
	Let $H^+:=\langle\Gamma_1,\,\xi\rangle$, and note that $H\leqslant H^+$.
	Firstly we find some invariant for elements of $H^+$ that differentiates $\xi$ from $H$. Note that all automorphisms in $H^+$ map blue edges to blue edges, preserving the orientations given by the arrows, and they also map red edges to red edges - but in this case they might flip the orientation. We will say that a blue edge $e$ is \textit{twisted} by an automorphism $h\in H^+$ if the red edges incident at one end of $e$ have orientation preserved by $h$, but the red edges incident at the other end have orientation flipped by $h$. The invariant of study will be the set of twisted edges, denoted $\tau(h)$. For instance, $\tau(g)=\emptyset$ for $g\in\Gamma_1$, and $\tau(\xi)=\{e,\overline{e}\}$ where $e$ is as indicated in the picture. Edges in $\tau(h)$ will always come in $\{e,\overline{e}\}$ pairs, so we will think of $\tau(h)$ as a set of geometric edges.
	
		\begin{figure}[H]
		\definecolor{RED}{HTML}{FF0000}
		\definecolor{BLUE}{HTML}{0000FF}
		\definecolor{ORANGE}{HTML}{FF6600}
		\definecolor{PURPLE}{HTML}{800080}
		\centering
		\begin{overpic}[width=.7\textwidth,clip=true]{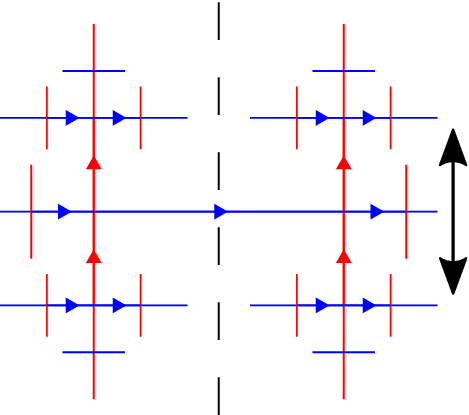}
			\put(40,45){\Huge\textcolor{BLUE}{$e$}}
			\put(101,40){\Huge$\xi$}
			\put(-5,75){\Huge$T$}	
		\end{overpic}
		\caption{Tree $T$ and automorphism $\xi$.}\label{fig:xi}
	\end{figure}
	
	Twisting twice about an edge results in an untwisted edge, so the set of twisted edges of a composition can be expressed as the following symmetric difference.
	\begin{equation}\label{taucomp}
	\tau(h_1h_2)=h_2^{-1}(\tau(h_1))\triangle\tau(h_2)
	\end{equation} 
	It immediately follows that $\tau(h)$ is finite for all $h\in H^+$, and that $\tau(h)$ contains an even number of geometric edges if and only if $h$ contains an even number of $\xi^\pm$ terms when expressed as a product of $\Gamma_1$ terms and $\xi^\pm$ terms. In particular, $\tau(h)$ contains an even number of geometric edges if $h\in H$.
	
	Now suppose for contradiction that $h\in H$ makes $\Gamma_1^h$ and $\Gamma_2$ commensurable. We know that $\tau(h)$ is finite, so suppose it spans a subtree of $T$ with diameter $M$. By our commensurability assumption, there exists $g\in\Gamma_1$ with translation length greater than $2M$ such that $h^{-1} \xi^{-1}g\xi h\in\Gamma_1$. By (\ref{taucomp}), $\tau(\xi^{-1}g\xi)=\{e,\overline{e},g^{-1}(e),g^{-1}(\overline{e})\}$ contains precisely two geometric edges, and these are at least $2M$ apart. Now $\tau(h)$ consists of an even number of geometric edges, pairwise less than $M$ apart, so again using (\ref{taucomp}) we deduce that $\tau(\xi^{-1}g\xi h)$ contains a pair of geometric edges separated by at least $M$ (it may contain other edges). A final application of (\ref{taucomp}) reveals that $\tau(h^{-1} \xi^{-1}g\xi h)$ is non-empty, contradicting $h^{-1} \xi^{-1}g\xi h\in\Gamma_1$.
\end{exmp}

\bigskip
\section{Bounds on sizes of covers}\label{sec:Bounds}

Leighton's Theorem assures the existence of a common finite cover but gives no bound on how large this cover might be. In Leighton's original paper \cite{Leighton} there is a short remark at the end saying that, given finite graphs $G_1$ and $G_2$, one can easily calculate an upper bound for the size of the common finite cover constructed - but this bound is not explicitly in terms of the number of edges and vertices in $G_1$ and $G_2$. In this section we obtain an explicit upper bound for the size of the finite cover constructed in Leighton's proof (one could alternatively obtain a bound from our new proof of Leighton's Theorem, but such a bound turns out to be larger); we also find an upper bound for the finite cover in Theorem \ref{thm:SymObjectLeighton} and a bound for the index of commensurability in Theorem \ref{thm:SymLeighton}. We make no claim that these bounds are anywhere close to being sharp.\par 
A key tool in finding bounds will be \textit{Landau's function} $g(n)$, which is the greatest order of an element of the symmetric group on $n$ elements. Equivalently, this is the greatest possible value for the lowest common multiple of positive integers $n_1,...,n_k$ that sum to $n$. \cite{Landau} gives the following explicit bound.
\begin{equation}\label{Landau}
g(n)\leq\exp(1.05313\sqrt{n\log{n}})
\end{equation}
As the other bounds we use will be quite rough, we will use the neater but less accurate bound $\exp(2\sqrt{n\log{n}})$ in place of (\ref{Landau}) - the important thing is that it's sub-exponential. The bounded version of Leighton's Theorem is as follows - note the notation we use here is consistent with \cite[Theorem 1.1]{Neumann} rather than with the rest of our paper.
\bigskip

\begin{thm}(Bounded Leighton's theorem)\label{boundL}\\
Let $G$ and $G'$ be finite connected graphs. Set $E:=\tfrac{1}{2}|E(G)|$ and $V':=|V(G')|$. Then $G$ and $G'$ have a common cover $H$ with the following bound on its vertex set:
\begin{equation*}
|V(H)|\leq 2V'\exp(2\sqrt{E\log{E}}).
\end{equation*}
\end{thm}
\begin{proof}
	We use the proof of \cite[Theorem 1.1]{Neumann}, which is essentially the same as Leighton's original proof, just written more concisely. The common cover $H$ constructed in the proof has vertices indexed by tuples $(i,v,v',\alpha)$, where $v$ and $v'$ are vertices of colour $i$ in $G$ and $G'$ respectively, and $\alpha\in A_i$. There are $n_i$ (resp. $n'_i$) vertices in $G$ (resp. $G'$) of colour $i$, and $|A_i|=s/n_i$, so $|V(H)|=\sum_in_in'_i(s/n_i)=V's$. Now $s$ is a common multiple of the $m_k$, where $m_k$ is the number of edges coloured $k$ in $G$. Note that $\{e\in E(G)\mid e\text{ has colour }\overline{k}\}=\{e\in E(G)\mid \overline{e}\text{ has colour }k\}$, so $m_k=m_{\overline{k}}$ and $m_k$ is even if $k=\overline{k}$. This implies that $s\leq2g(E)\leq2\exp(2\sqrt{E\log{E}})$.
\end{proof}

\begin{remk}
	One can obtain a much stronger bound in the case that $G_1$ and $G_2$ are both regular, namely $|V(G)|\leq |V(G_1)||V(G_2)|$ if they are regular of even degree, and $|V(G)|\leq 2|V(G_1)||V(G_2)|$ if they are regular of odd degree.\par 
	 A classical theorem of Petersen says that any $2d$-regular graph has a \textit{2-factor} (the edges of a 2-regular subgraph that contains every vertex of the graph); so by induction one can partition the edge set into $d$ 2-factors, which gives us the data of a covering of $R_d$, the \textit{rose on $d$ petals} (graph with one vertex and $d$ geometric edges). See \cite[Corollary 2.1.5]{Diestel} for a proof of Petersen's Theorem - note that this proof is written for the setting of simplicial graphs, but the same proof works for our definition of graph (allowing multi-edges and loops). If $G_1$ and $G_2$ are both $2d$-regular, then we have covering maps $G_1\to R_d$ and $G_2\to R_d$, corresponding to subgroups $\pi_1 G_1, \pi_1 G_2\leqslant\pi_1 R_d$ of index $|V(G_1)|$ and $|V(G_2)|$ respectively. The subgroup $\pi_1 G_1\cap \pi_1 G_2\leqslant\pi_1 R_d$ then has index at most $|V(G_1)||V(G_2)|$, so corresponds to a common cover of $G_1$ and $G_2$ with at most $|V(G_1)||V(G_2)|$ vertices.\par 
	 If $G_1$ and $G_2$ are $d$-regular with $d$ odd, then instead of the rose we can use the graph $P_d$ consisting of two vertices and $d$ geometric edges joining them. $G_1$ might not be a cover of $P_d$ because it might have cycles of odd length, but this turns out to be the only obstruction. Indeed we can take a double cover $\hat{G}_1$ of $G_1$ that only contains even length cycles, hence is bipartite, and by Hall's Matching Theorem (or more specifically \cite[Corollary 2.1.3]{Diestel}) there exists a complete matching (or 1-factor) in $\hat{G}_1$; then by induction we can partition the edge set of $\hat{G}_1$ into $d$ complete matchings, and this is precisely the data of a covering of $P_d$. Similarly $G_2$ has a double cover $\hat{G}_2$ that covers $P_d$. Then we can take a cover $G$ of $P_d$ corresponding to the subgroup $\pi_1\hat{G}_1\cap\pi_1\hat{G}_2\leqslant\pi_1 P_d$ of index at most $|V(G_1)||V(G_2)|$, and this will be a common cover of $\hat{G}_1$ and $\hat{G}_2$, and hence also of $G_1$ and $G_2$, and $|V(G)|\leq 2|V(G_1)||V(G_2)|$ as desired.
	 
\end{remk}
\bigskip

The bounded version of Theorem \ref{thm:SymObjectLeighton} is the following.

\begin{thm}(Bounded Graph of Objects Leighton's Theorem)\\\label{thm:boundedobj}
	Let\begin{align*}
	f^1:\tilde{X}\to X^1\\
	f^2:\tilde{X}\to X^2
	\end{align*}
	be coverings of graphs of objects, with $G_1:=G_{X^1}$ and $G_2:=G_{X^2}$ both finite, and $T:=G_{\tilde{X}}$ a tree. Let $\Gamma_1$ and $\Gamma_2$ be the groups of deck transformations for $f^1$ and $f^2$, and suppose that $\Gamma_1,\Gamma_2\leqslant H\leqslant\Aut(\tilde{X})$.
	Suppose also that $H$ has finite edge isotropy groups. Then $X^1$ and $X^2$ have a common finite cover $X$, and there exists $g\in\mathscr{S}(H)$ that fits into the following commutative diagram of coverings.
	\begin{equation*}
	\begin{tikzcd}[
	ar symbol/.style = {draw=none,"#1" description,sloped},
	isomorphic/.style = {ar symbol={\cong}},
	equals/.style = {ar symbol={=}},
	subset/.style = {ar symbol={\subset}}
	]
	\tilde{X}\ar{dd}[swap]{f^1}\ar{rr}{g}\ar{dr}&&\tilde{X}\ar{dl}\ar{dd}{f^2}\\
	&X\ar{dl}{\mu^1}\ar{dr}[swap]{\mu^2}\\
	X^1&&X^2
	\end{tikzcd}
	\end{equation*}
	Furthermore, $X$ has underlying graph $G$ with vertex set bounded as follows.
	\begin{equation*}
	|V(G)|\leq (d!)^2(\LCM_{e\in E(T)}|H(e)|)^{2d}V^2\exp(2\sqrt{V\log{V}}),
	\end{equation*}
	where $V:=|V(G_1)\sqcup V(G_2)|$, $d$ is the maximum degree of vertices in $T$, and $\LCM$ denotes the lowest common multiple.\\ (As the isomorphism-type of the group $H(e)$ only depends on the $H$-orbit of $e$, and there are finitely many of these orbits, the lowest common multiple of the $|H(e)|$ is guaranteed to be finite.)
\end{thm}
\begin{proof}
	We follow the proof of Theorem \ref{thm:SymObjectLeighton}, and recall that the vertex set of $G$ is defined by
	\begin{equation*}
	V(G):=\left\{(s,l)\mid s\in\calS(u_1,u_2),\,u_1\in V(G_1),\,u_2\in V(G_2),\,1\leq l\leq\frac{N}{|\calS(u,-)|}\right\},
	\end{equation*}
	where $N$ is a common multiple of all the integers $|\calS(u,-)|$ and $|A(e)|$ for $u\in V(G_1)$ and $e\in E(G_1)$. Observe that by Lemma \ref{orbstab}, $|\calS(u,-)|=|\text{Stab}_\calS(e,e,1_e)||\calS\cdot (e,e,1_e)|$ if $\partial_0e=u$, and $\calS\cdot (e,e,1_e)=A(e)$ by Step 2 of the proof of Theorem \ref{thm:SymObjectLeighton}. So it suffices to make $N$ a common multiple of the integers $|\calS(u,-)|$. Also by Lemma \ref{orbstab}, we have that
	\begin{equation*}
	|\calS(u,-)|=|\calS(u,u)||\{v\mid\calS(u,v)\neq\emptyset\}|.
	\end{equation*}
	The sets $\{v\mid\calS(u,v)\neq\emptyset\}$ are the components of $\calS$, so partition $V(G_1)\sqcup V(G_2)$, and we can use (\ref{Landau}) to bound the lowest common multiple of their orders.\\
	\begin{claim}
$|\calS(u,u)|$ divides $d!(\text{LCM}_{e\in E(T)}|H(e)|)^{d}$.
	\end{claim}\\
\begin{claimproof}
The claim in step 1 of the proof of Theorem \ref{thm:SymObjectLeighton} says that there is a homomorphism $\theta:\calS(u,u)\to$ Aut(star$(u))$. The image of $\theta$ will have order dividing $d!$, and the claim goes on to show that each $s\in\ker\theta$ can be written as
\begin{equation*}
s=(f^1)^{z_0}h^{z_0}((f^1)^{z_0})^{-1}
\end{equation*}
for some $z_0$ a lift of $u$ to $T$, and $h\in H$ that fixes $\star*(z_0)$. We have $h_e\in H(e)$ for each $e\in$ star$(z_0)$, so we get an injective homomorphism $\ker\theta\to\prod_{e\in \text{star}(z_0)}H(e)$, which completes the proof of the claim.
\end{claimproof}\\
	
	Putting this together we get the following bound on $N$.
	\begin{equation}\label{Nbound3}
	N\leq d!(\text{LCM}_{e\in E(T)}|H(e)|)^{d}\exp(2\sqrt{V\log{V}}).
	\end{equation}
	It remains to count $s$ such that $s\in\calS(u_1,u_2)$ for $u_1\in V(G_1)$ and $u_2\in V(G_2)$. For fixed $u_1$ and $u_2$, $|\calS(u_1,u_2)|=|\calS(u_1,u_1)|$, and so the total count of these $s$ will be at most $d!(\text{LCM}_{e\in E(T)}|H(e)|)^{d}V^2$. Combining this with (\ref{Nbound3}) proves the theorem.
\end{proof}
\bigskip

Theorem \ref{thm:boundedobj} can be used to obtain a bounded version of the Symmetry-restricted Leighton's Theorem as follows.

\begin{thm}(Bounded Symmetry-restricted Leighton's Theorem)\\
	Let $T$ be a tree, and $H\leqslant\Aut(T)$, and let $\Gamma_1, \Gamma_2 \leqslant H$ be free uniform lattices in $\Aut(T)$.
	Then for all $R \in \mathbb{N}$ there exists $g \in \mathscr{S}_R(H)$ such that $\Gamma_1^g$ is commensurable to $\Gamma_2$ in $\Aut(T)$. Furthermore, we have the bound:
	$$|\Gamma_1^g:\Gamma_1^g\cap\Gamma_2|\leq(d!)^{4d^{R-1}}V^2\exp(2\sqrt{V\log{V}}),$$	
	where $V:=|V(T/\Gamma_1)\sqcup V(T/\Gamma_2)|$ and $d$ is the maximum degree of vertices in $T$.
\end{thm}
\begin{proof}
In the proof of Theorem \ref{thm:SymLeighton2} we converted $T$ into a tree of objects $X$ such that we have an isomorphism $\psi:\Aut(T)\to\Aut(X)$ with $\psi(\mathscr{S}_R(H))=\mathscr{S}(\psi(H))$. We can therefore apply Theorem \ref{thm:boundedobj} to $\psi(\Gamma_1),\psi(\Gamma_2)\leqslant\psi(H)\leqslant\Aut(X)$ to obtain $g\in\mathscr{S}(\psi(H))$ such that $\psi(\Gamma_1)^g$ is commensurable to $\psi(\Gamma_2)$, and the bound on $|V(G)|$ from the theorem also serves as an upper bound on the index $|\psi(\Gamma_1)^g:\psi(\Gamma_1)^g\cap\psi(\Gamma_2)|$. We then have $g':=\psi^{-1}(g)\in\mathscr{S}_R(H)$ an element that conjugates $\Gamma_1$ to become commensurable to $\Gamma_2$, and we have the bound:
\begin{equation}\label{objtosymbound}
|\Gamma_1^{g'}:\Gamma_1^{g'}\cap\Gamma_2|\leq|V(G)|\leq(d!)^2(\LCM_{e\in E(T)}|\psi(H)(e)|)^{2d}V^2\exp(2\sqrt{V\log{V}}).
\end{equation}

So it remains to estimate the orders of the isotropy groups $\psi(H)(e)$. From the proof of Theorem \ref{thm:SymLeighton2} we know that the edge object $X_e$ is the $(R-1)$-neighbourhood $N_{R-1}(e)$ of the edge $e$. And $\psi(H)(e)$ is a subgroup of $\Aut(N_{R-1}(e),e)$, the group of automorphisms of $N_{R-1}(e)$ that fix $e$. If we embed $N_{R-1}(e)$ in a $d$-regular tree $Y$, sending $e\mapsto e'$, then we get a (non-unique) injective homomorphism $\Aut(N_{R-1}(e),e)\xhookrightarrow{}\Aut(N_{R-1}(e'),e')$. Specifying an element of $\Aut(N_{R-1}(e'),e')$ is equivalent to specifying, for each $v\in N_{R-2}(e')$, a permutation of the $d-1$ edges leaving $v$ that point away from $e$. There are
$$2(1+(d-1)+(d-1)^2+...+(d-1)^{R-2})=\tfrac{2}{d}((d-1)^{R-1}-1)$$
such vertices $v$, so we have
\begin{align*}
|\Aut(N_{R-1}(e'),e')|&=((d-1)!)^{\tfrac{2}{d}((d-1)^{R-1}-1)}\\
&\leq (d!)^{\tfrac{2}{d}(d^{R-1}-1)}.
\end{align*}
We know that $\LCM_{e\in E(T)}|\psi(H)(e)|\leq |\Aut(N_{R-1}(e'),e')|$, and so we can plug this estimate into (\ref{objtosymbound}), which completes the proof.
\end{proof}

\bigskip
\begin{appendices}
	\section{Alternative proof of symmetry-restricted version}
	\bigskip
	\begin{center}
		{\large Giles Gardam and Daniel J. Woodhouse}
	\end{center}

\bigskip

Let $T$ be a locally finite, simplicial tree with all edges identified with $[0,1]$.
Let $G = \Aut(T)$, the full simplicial automorphism group of $T$.
We will assume that $G$ acts without edge inversions, which is possible by either subdividing $T$ (provided $T$ is not isometric to $\mathbb{R}$), or passing to an index $2$ subgroup.
Thus we can make each edge a directed edge such that these orientations are preserved by the $G$-action.
Let $H \leq G$ be a subgroup such that $H \acts T$ cocompactly.

If $S \subseteq T$ is a finite subset of the vertices then the \emph{set-wise stabilizer} is $H_S  = {\{ h \in H \mid h \cdot S = S \}},$ and the \emph{pointwise stabilizer} is ${H_{(S)} =  \{ h \in H \mid  \textrm{ $h \cdot s = s$ \; for all \; $s \in S$} \}}.$

\bigskip

Recall the Symmetry-restricted Leighton's Theorem:

\begin{thm} \label{thm:A}
	Let $\Gamma_1, \Gamma_2 \leqslant H$ be free uniform lattices in $G$.
	Then for all $R \in \mathbb{N}$ there exists $g \in \mathscr{S}_R(H)$ such that $\Gamma_1^g$ is commensurable to $\Gamma_2$ in $G$. 
\end{thm}
\bigskip

Before proving the theorem we make some remarks about the $R$-symmetry-restricted closure $\mathscr{S}_R(H)$.

\begin{remk}
	It is easy to check that $\mathscr{S}_R(H) = \bigcap_{v \in V X} H G_{(B_R(v))}$, hence it is a closed subgroup of $G$. 
\end{remk}
\begin{remk}
Theorem \ref{thm:A} is a strengthening of the Bass--Kulkarni Uniform Commensurability Theorem \cite[4.8 (c)]{BassKulkarni}, since $G_H$, defined to be the largest subgroup of $G$ preserving all $H$-orbits, contains $\mathscr{S}_1(H)$.
\end{remk}

\begin{remk}\label{notHbar}
	Although $\cap_{R \in \N} \mathscr{S}_R(H) = \overline{H}$, the closure of $H$, we cannot necessarily find such a conjugating element $g$ as in Theorem~\ref{thm:A} with $g \in \overline{H}$.
	An example pointed out to us by Alexander Lubotzky is $H = \operatorname{SL}_2(\mathbb{Q}_p)$ (or in general a simple rank 1 Lie group over a local non-archimedean field) acting on its Bruhat--Tits tree, in which case $H$ is closed and there are uncountably many $H$-conjugacy classes of uniform lattices in $H$ \cite{Lubotzky}.
	The commensurability class of a uniform tree lattice, however, is countable: in general, every subgroup $\Gamma_2$ commensurable with a given subgroup $\Gamma_1 \leq G$ is contained in the commensurator subgroup $\operatorname{Comm}_G(\Gamma_1)$ (since for every $g \in \Gamma_2$ we have $\Gamma_1^g$ commensurable with $\Gamma_2^g = \Gamma_2$ and thus with $\Gamma_1$), and for the case of $\Gamma_1$ a uniform tree lattice the commensurator is countable by \cite[Corollary (8.6), p. 885]{BassKulkarni} and thus has only countably many finitely generated subgroups.
	Thus, there are uniform lattices not commensurable up to conjugacy in $H$.
	In such examples the finite index in the commensuration achievable by $g \in \mathscr{S}_R(H)$ tends to $\infty$ as $R \to \infty$.
	This is necessarily the case, since the set of $g \in G$ such that $[\Gamma_2 : \Gamma_1^g \cap \Gamma_2] \leqslant N$ is the union of \emph{finitely many} cosets of $\Gamma_1$ (since there are finitely many subgroups of finite index, and only finitely many isomorphisms between finite quotient graphs).
	If the index did not tend to infinity, the intersection of the nested sets of conjugating $g \in \mathscr{S}_R(H)$ would be non-empty, a contradiction.
\end{remk}

The following proof is essentially an adaptation of the ideas in~\cite{Woodhouse}.

\subsection{Proof of Theorem~\ref{thm:A}}

Let $\Gamma_1, \Gamma_2 \leq H \leq G$, be free uniform lattices, as above, with $X_i = T / \Gamma_i$.
Given a tree $K$ and a simplicial map $f: K \rightarrow X_i$ we will let $\tilde{f}: K \rightarrow T$ denote some choice of lift of $f$.\par

An \emph{$R$-polyhedron} in $T$ is the closed $R$-neighborhood of a vertex $v$, written $B_R(v)$.
We say that $v$ is the \emph{center} of the polyhedron.
%
A \emph{polyhedral pair} $\mathbf{P} = (P, \phi_1, \phi_2)$ over $X_1$ and $X_2$ is a graph $P$ with simplicial maps $\phi_i : P \rightarrow X_i$ such that the lift $\tilde{\phi}_i$ embeds $P$ in $T$ as an $R$-polyhedron.
The vertex in $P$ that maps to the center of $\tilde{\phi}_i(P)$ is the \emph{center} of $\mathbf{P}$.
An \emph{$H$-admissible polyhedral pair} $\mathbf{P} = (P, \phi_1, \phi_2)$  is an $R$-polyhedral pair such that there exists $h \in H$ such that $h \circ \tilde{\phi}_1 = \tilde{\phi}_2$. 
Note that this does not depend on the choice of lifts $\tilde{\phi}_1, \tilde{\phi}_2$ since they differ by deck transformations.

An \emph{$R$-face} in $T$ is the closed $(R-1)$-neighborhood of an edge $e$ in $T$.
We say that $e$ is the \emph{center} of the face.
An \emph{$R$-face pair} $\mathbf{F} = (F, \varphi_1, \varphi_2)$ over $X_1, X_2$ is a simplicial graph $F$ with simplicial maps $\varphi_i: F \rightarrow X_i$  such that the lift $\tilde{\varphi}_i$ embeds $F$ in $T$ as an $R$-face.
The edge in $F$ that maps to the center of $\tilde{\varphi}_i(F)$ is the \emph{center} of $\mathbf{F}$.
An \emph{$H$-admissible face pair} $\mathbf{F} = (F, \varphi_1, \varphi_2)$  is an $R$-face pair such that there exists $h \in H$ such that $h \circ \tilde{\varphi}_1 = \tilde{\varphi}_2$. 
Note that this does not depend on the choice of lifts $\tilde{\varphi}_1, \tilde{\varphi}_2$ as different lifts only differ by deck transformation.

The faces of a polyhedral pair $\mathbf{P} = (P, \phi_1, \phi_2)$ are the face pairs you obtain by restricting to the $(R-1)$-neighborhood of an edge incident to the center of $\mathbf{P}$.
Recall that the edges of $T$ are directed and these orientations are $G$-equivariant, and hence the edges of $X_i$  are also directed.
If $\mathbf{F}$ is an $H$-admissible $R$-face, then the edge $e$ of $F$ has a direction determined by the map $\phi_1$ that is consistent with the orientation determined by the map $\phi_2$.
We say that $\mathbf{P}$ is either on the left or on the right of its face $\mathbf{F}$ depending on this orientation.

If $\mathbf{P}$ and $\mathbf{P}' = (P', \phi_1', \phi_2')$ are polyhedral pairs that are respectively on the left and on the right of the face $\mathbf{F}$, then we may glue them together by identifying $P$ and $P'$ along the subspace $F$.
More precisely, there exist subspaces $F \subset P$ and $F' \subset P'$, that give the restrictions of $\mathbf{P}$ and $\mathbf{P}'$ to $\mathbf{F}$, and an isomorphism $\theta: F \rightarrow F'$ such that $ \phi_i=\phi'_i \circ \theta$.
Then defining $P \cup P'$ to be the quotient space $P \sqcup P' / \sim$ where $p \sim p'$ if and only if $\theta(p) = p'$.
The resulting space $P \cup P'$ has well defined maps $P \cup P' \rightarrow X_i$ since $\phi_i(p) = \phi_i'(\theta(p))$ for all $p \in F$.
Since $\mathbf{P}$ is on the left of $\mathbf{F}$ and $\mathbf{P}'$ is on the right,
the maps $\phi_1, \phi_1'$ lift to embeddings $\tilde{\phi}_1, \tilde{\phi}_1'$ of $P$ and $P'$ as $R$-polyhedrons in $T$ centered at vertices $\tilde{v}$ and $\tilde{v}'$ adjacent along an edge $\tilde{e}$. 
Thus, the union $ P\cup P'$ immerses in $X_i$, since it lifts to an embedding in $T$ as the $R$-neighborhood of an edge $\tilde{e}$.
We say that we have \emph{glued $\mathbf{P}$ to $\mathbf{P}'$ along $\mathbf{F}$}.

Let $\mathcal{P}$ denote the set of all $H$-admissible $R$-polyhedral pairs over $X_1$ and $X_2$.
If $\mathbf{F}$ is an $H$-admissible $R$-face pair then let $\overleftarrow{\mathbf{F}}$ denote the set of all $\mathbf{P} \in \mathcal{P}$ on the left of $\mathbf{F}$, and define $\overrightarrow{\mathbf{F}}$ similarly.

\begin{lem} \label{lem:cosetCorrespondance}
	Let $\mathbf{F} = (F, \varphi_1, \varphi_2)$ be an $H$-admissible $R$-face pair and suppose that $\mathbf{P} = (P, \phi_1, \phi_2) \in \overleftarrow{\mathbf{F}}$.
	Then $$\big|\overleftarrow{\mathbf{F}} \big| = [ H_{(\widetilde{F})} : H_{(\widetilde{P})} ],$$
	where $\widetilde{F} = \tilde{\varphi}_1(F )$ and $\widetilde{P} = \tilde{\phi}_1(P)$ are chosen so that $\widetilde{F} \subset \widetilde{P}$.
\end{lem}

\begin{proof}
	Let $\widetilde{F}_i \subseteq T$ be the $R$-face $\tilde{\varphi}_i(F)$ and let $e_i$ be the center of $\widetilde{F}_i$ and $v_i$ be the vertex on the left of $e_i$.
	Let $\widetilde{P}_i \subseteq T$ be the $R$-polyhedron $\tilde{\phi}_i(P)$, where we have chosen the lift so that $\widetilde{F}_i \subset \widetilde{P}_i$.
	Observe that since $\mathbf{P}$ is on the left of $\mathbf{F}$, the center of $\widetilde{P}_i$ is $v_i$.
	By $H$-admissibility, there exists $h \in H$ such that $h \circ \tilde{\phi}_1 = \tilde{\phi}_2$, so $h \cdot \widetilde{P}_1 = \widetilde{P}_2$.
	
	Conversely, any $h' \in H$ such that $h' \cdot \widetilde{P}_1 = \widetilde{P}_2$ defines an $H$-admissible $R$-polyhedral pair $\mathbf{P}' = (P' , \phi_1', \phi_2')$ by giving an identification of $\widetilde{P}_1$ to $\widetilde{P}_2$.
	Indeed, we can see that any $\mathbf{P}' \in \overleftarrow{\mathbf{F}}$ can be obtained from some such identification by an element $h' \in H$.

	Note that $\mathbf{P}' \in \overleftarrow{\mathbf{F}}$ if and only if $h^{-1}h' \in H_{(\widetilde{F}_1)}$.
	Moreover, $\mathbf{P}' = \mathbf{P}$ if and only if $h^{-1}h' \in H_{(\widetilde{P}_1)}$.
	Thus we deduce that the elements in $\overleftarrow{\mathbf{F}}$ correspond to the elements in $H_{(\widetilde{F})} / H_{(\widetilde{P})}$.
\end{proof}
\bigskip

We wish to find a non-trivial weight function $\omega: \mathcal{P} \rightarrow \mathbb{N}$ such that for all $H$-admissible $R$-faces $\mathbf{F}$, we have

\[
\sum_{\mathbf{P} \in \overleftarrow{\mathbf{F}}} \omega(\mathbf{P}) = \sum_{\mathbf{P} \in \overrightarrow{\mathbf{F}}} \omega(\mathbf{P})
\]

We say that such $\omega$ satisfy the \emph{basic gluing equations}.
\bigskip

If $\omega$ satisfies the basic gluing equations, then by taking $\omega(\mathbf{P})$ copies of each $\mathbf{P}$ we can glue them together to obtain a graph $\hat{X}$ that covers both $X_1$ and $X_2$, whose restriction to $R$-polyhedrons give $H$-admissible $R$-polyhedral pairs.

Let $\mu$ be the bi-invariant Haar measure for $H$.
Let $\omega(\mathbf{P}) = \mu(H_{(\widetilde{P})})$ where $\widetilde{P} = \tilde{\phi}_i(P)$, so the stabilizer is well defined up to conjugation in $H$, and hence the Haar measure is well defined (and doesn't depend on the choice of lift or $i \in \{ 1,2\}$).
If $\widetilde{P}$ and $\widetilde{P}'$ are $R$-polyhedrons in $T$, and let $K$ denote their union, then $H_{(K)}$ is a finite index subgroup of both $H_{(\wt{P})}$ and $H_{(\wt{P}')}$.
Thus the Haar measures of the pointwise stabilizers are commensurable: 
\[
\frac{\mu(H_{(\wt{P})})}{\mu(H_{(\wt{P}')})} = \frac{[H_{(\wt{P})} : H_{(K)}] }{ [H_{(\wt{P}')} : H_{(K)}]} 
\]
Since there are only finitely many orbits of $R$-polyhedron in $T$ it follows that we can scale $\omega$ so that it takes integer values.

\begin{lem} \label{lem:basic}
	The weight function $\omega$ satisfies the basic gluing equations.
\end{lem}

\begin{proof}
	We first observe that since all $\mathbf{P} \in \overleftarrow{\mathbf{F}}$ can have lifts $\tilde{\phi}_1$ chosen so that they map $P$ to the same $R$-polyhedron $\widetilde{P}$ in $T$, their weights are all equal. 
	This implies
	\begin{align*}
	\sum_{\overleftarrow{\mathbf{F}}} \omega(\mathbf{P}) & = \mu(H_{(\widetilde{P})}) |\overleftarrow{\mathbf{F}} | \\
	& = \mu(H_{(\widetilde{P})}) [H_{(\widetilde{F})} : H_{(\widetilde{P})}] \\
	& = \mu(H_{(\widetilde{F})}). \\
	\end{align*}
	where the second equality follows from Lemma~\ref{lem:cosetCorrespondance}.
	So the left and right side of the equations are equal.
\end{proof}
\bigskip

\begin{proof}[Proof of Thm~\ref{thm:A}]
	Let $\omega$ be the integer valued weight function satisfying the basic gluing equations (Lemma~\ref{lem:basic}). 
	Let $\mathcal{P}_\omega$ be the set obtained by taking $\omega(\mathbf{P})$ copies of each $\mathbf{P} \in \mathcal{P}$, for each $H$-admissible $R$-face pair $\mathbf{F}$ we can choose a one to one correspondence between the polyhderal pairs in $\mathcal{P}$ on the left of $\mathbf{F}$ and those on the right of $\mathbf{F}$ and glue them together.
	Corresponding faces can be glued together.
	By gluing all corresponding faces over all face pairs we obtain a common cover $\hat{X}$ of $X_1$ and $X_2$.
	
	To check that $\hat{X}$ is indeed a covering space of both $X_1$ and $X_2$, observe that if a vertex $v$ in $\hat{X}$ is the center of some $\mathbf{P} \in \mathcal{P}_\omega$ then $\hat{X}$ is locally a common cover at that point.
	Moreover, any vertex $u$ that is adjacent to $v$ is also the center of some other $\mathbf{P}' \in \mathcal{P}_\omega$ since each $R$-face of $\mathbf{P}$ has an $R$-polyhedral pair glued to it.
	Thus we can inductively conclude that $\hat{X}$ is locally a common cover at all vertices in the same connected component as $v$, and hence that component is a common cover of $X_1$ and $X_2$.
	Since every component of $\hat{X}$ is contructed from elements of $\mathcal{P}_\omega$, all components are common covers.
	
	Let $f:T \rightarrow T$ be an automorphism (unique up to pre and post composition by deck transformations) such that the following diagram commutes:
	\begin{equation*}
	\begin{tikzcd}[
	ar symbol/.style = {draw=none,"#1" description,sloped},
	isomorphic/.style = {ar symbol={\cong}},
	equals/.style = {ar symbol={=}},
	subset/.style = {ar symbol={\subset}}
	]
	T\ar{dd}\ar{rr}{f}\ar{dr}&&T\ar{dl}\ar{dd}\\
	&\hat{X}\ar{dl}\ar{dr}\\
	X_1&&X_2
	\end{tikzcd}
	\end{equation*}	
	Then $f \in \mathscr{S}_R(H)$ since $\hat{X}$ is constructed from $H$-admissible $R$-polyhedral pairs. Indeed, each $R$-polyhedron $\wt{P}$ in $T$ determines a polyhedral pair $\mathbf{P} = (P, \phi_1, \phi_2)$ given by its image in $\hat{X}$.
	Then $\phi_1$ and $\phi_2$ have lifts such that $f \circ \tilde{\phi}_1 = \tilde{\phi}_2$ and with $h \in H$ such that $h \circ \tilde{\phi}_1 = \tilde{\phi}_2$, so $h$ is equal to $f$ on $\wt{P}$.
\end{proof}

\end{appendices}

\bigskip
\bigskip

\emph{Email addresses:}\\
samuel.shepherd@vanderbilt.edu\\
ggardam@uni-muenster.de\\
daniel.woodhouse@maths.ox.ac.uk

\end{document}